\documentclass{amsart}
\usepackage{amsthm,amssymb,amsmath,dsfont,bigdelim}
\usepackage{graphicx}
\usepackage{float}
\usepackage[arc,all]{xy}
\usepackage{longtable}
\usepackage[margin=0pt,font=small,labelfont={normalsize},labelsep=space]{caption}
\newtheorem{thm}{Theorem}[section]

\newtheorem{lem}[thm]{Lemma}
\newtheorem{prop}[thm]{Proposition}
\theoremstyle{definition}

\newtheorem*{proof M}{\textbf{Proof of  Main Theorem}}
\theoremstyle{remark}

\numberwithin{equation}{section}

\begin{document}
\title[ ]{On characterizing nilpotent Lie algebra by their multiplier, $ s(L)=6, 7 $}%
\author{afsaneh shamsaki}%
\address{School of Mathematics and Computer Science\\
Damghan University, Damghan, Iran}
\email{Shamsaki.afsaneh@yahoo.com}%
\author[P. Niroomand]{Peyman Niroomand}\thanks{Corresponding Author: Peyman Niroomand}
\email{niroomand@du.ac.ir, p$\_$niroomand@yahoo.com}
\address{School of Mathematics and Computer Science\\
Damghan University, Damghan, Iran}

\begin{abstract}
Let $ L $ be an $ n $-dimensional non-abelian nilpotent Lie algebra and $ s(L)=\frac{1}{2}(n-1)(n-2)+1-\dim \mathcal{M}(L) $ where $ \mathcal{M}(L) $ is the Schur multiplier of a Lie algebra   $ L. $ The structures of nilpotent Lie algebras $ L $ when $ s(L)\in \lbrace 0,1,2,3,4,5\rbrace $ are determined. In this paper, we classify all non-abelian nilpotent Lie algebras $ L $ when $ s(L)=6,7. $ 
\end{abstract}
\maketitle
\section{ Introduction and preliminaries}
Let $L$ be a Lie algebra on a fixed field $\mathbb{F}$ and  $L\cong F/R$  for a free Lie algebra. Then the Schur multiplier $ \mathcal{M}(L) $ of  $ L $ is isomorphic to $ R\cap F^{2}/[R, F] $. Classification of finite dimensional nilpotent Lie algebras according to the dimension of its Schur multiplier has been a line of investigation  by many authors in \cite{B1, B2,  M}. One of the main themes of the Lie theoretical part of research on the Schur multipliers of Lie algebras has been to determine all nilpotent Lie algebras according to their Schur multipliers. It is know that from \cite[Lemma 2.2]{M} that for a nilpotent Lie algebra $ L $, we have   $\dim \mathcal{M}(L)= 1/2n(n-1)-t(L)$ for an integer $t(L)\geq 0$. All  finite dimensional nilpotent Lie algebras $L$ with   $t(L) = 0, 1,\dots,7,8$  have been classified by several papers in \cite{B1,H1,H2}.  The second author, as a joint paper,  showed that the dimension of Schur multiplier of a non-abelian nilpotent Lie algebra is equal to $1/2(n-1)(n-2)+1-s(L)$ for an integer $s(L)\geq 0$. Similarly, we may classify  nilpotent Lie algebras $L$ rely on $s(L)$. 
The classification nilpotent Lie algebras in term of $s(L)=0,\ldots ,5$ are given in \cite{N, s1, s3, s4, s5}. This classification  simplify the problem of determining a Lie algebra in terms of $t(L)$ and recently are used to answer the same question for an $n$-Lie super algebra and Leibniz $ n $-algebras   in \cite{sa1, sa2}.
  In this paper, we are going to give a  classification of all nilpotent Lie algebras $ L $  for $t(L) =6, 7$.\\
The following theorems and lemmas are useful in the next sections.
\begin{thm} \cite[Theorem 3.1]{N}\label{th1.1}
Let $L$ be an $n$-dimensional nilpotent Lie algebra and $\dim L^{2}=m\geq 1$. Then
\begin{align*}
\dim\mathcal{M}(L)\leq \dfrac{1}{2}(n+m-2)(n-m-1)+1.
\end{align*}
Moreover, if $m=1$, then the equality holds if and only if $L\cong H(1)\oplus A(n-3)$.
\end{thm}
For a Lie algebra $ X, $ let $X^{ab}$ be used to denote $X/X^2.$ 
\begin{thm}\cite[Corollary 2.3]{N}\label{th1.2}
Let $L$ be a finite dimensional Lie algebra and $K$ be a central ideal of $L$. Then
\begin{align*}
\dim\mathcal{M}(L)+\dim(L^{2}\cap K) \leq \dim\mathcal{M}(L/K)+\dim\mathcal{M}(K)+\dim ((L/K)^{ab}\otimes K) .
\end{align*}
\end{thm}
For a Lie algebra $ L $, the concept epicenter $  Z^{*}(L) $ is introduced in \cite{sal}. The importance of  $  Z^{*}(L) $ is due to notion of  the fact that $ L $ is capable if and only if $  Z^{*}(L)=0 $.  Another notion having relation  to the capability is the concept of the exterior center of a Lie algebra $  Z^{\wedge}(L) $ which is introduced in \cite{N1}.  It is known that from   \cite[Lemma 3.1]{N1}, $ Z^{*}(L)=Z^{\wedge}(L) $.
\begin{lem}\cite[Corollary 2.3]{s4}\label{l1.1}
    Let $L$ be a  non-capable nilpotent  Lie algebra of dimension $n$ such that $ \dim L^{2}\geq 2 $. Then
\begin{equation*}
n-3<s(L).
\end{equation*}
\end{lem}
The following Lemma determines the structure of all nilpotent Lie algebras $ L $ when $ \dim L^{2}=1. $
\begin{lem}\cite[Lemma 3.3]{s1}\label{l1.2}
Let $ L $ be an $ n $-dimensional Lie algebra with derived subalgebra of dimension $ 1. $  Then $ L $ is isomorphic to $ H(m)\oplus A(n-2m-1) $ for some $ m\geq 1. $ 
\end{lem} 
The Lie algebra $ L $ is a central product of $ A $ and $ B, $ if $ L=A+B, $ where $ A $ and $ B $ are ideals of $ L $  such that $ [A,B]=0 $ and $ A\cap B\subseteq Z(L). $ We denote the central product of two Lie algebras $ A $ and $ B $ by $ A\dotplus B. $\\
We say a Lie algebra $ L $ is semidirect sum of an ideal $ I $ and subalgebra $ K $ if $ L=I+K $ and $ I\cap K=0. $ The semidirect sum of an ideal $ I $ and a subalgebra $ K $ is denoted by $ I\rtimes K. $
\begin{lem}\cite[Lemma 1]{I}\label{li}
Let $ L $ be a nilpotent Lie algebra and $ H $ be a subalgebra of $ L $ such that $ L^{2}=H^{2}+L^{3}. $ Then $ L^{i}=H^{i} $ for all $ i\geq 2. $ Moreover, $ H $ is an ideal of $ L. $ 
\end{lem}
Let $ L $ be a Lie algebra we assume that the reader is familiar with the basic definitions and properties of the tensor square $L\otimes L$ in \cite{E}.
   By a terminology of \cite{E}, let $ L\square L $ be the submodule of tensor square $ L\otimes L $ generated by the elements $ l\otimes l $ for all $ l\in L $. One can check that $ L\square L $ lies in the center of $ L\otimes L $ and  the exterior square $ L\wedge L $ of $ L $ is defined as   $  L\otimes L / L\square L$.\\
By looking at the classification of all nilpotent Lie algebras in \cite{G, M}, the nilpotent Lie algebras with derived subalgebra of dimension $ m $ such that $ 2\leq m\leq 4 $ are listed in the following tables.
\begin{longtable}{cccc}
\caption{: nilpotent Lie algebras of dimension at most $ 6 $ with $  \dim L^{2}=2$}\\
\hline
 \multicolumn{1}{c}{\textbf{Name}} & \multicolumn{1}{c}{\textsf{Nonzero multiplication}}    \\
\hline
\endhead
\hline \multicolumn{2}{r}{\small \itshape continued on the next page}
\endfoot
\endlastfoot
$ L_{4,3} $& $[x_{1},x_{2}] =x_{3}, [x_{1},x_{3}] =x_{4}$\\ \\
$ L_{5,3} $& $[x_{1},x_{2}] =x_{3}, [x_{1},x_{3}] =x_{4}$\\ \\
$ L_{5,5} $& $[x_{1},x_{2}] =x_{3}, [x_{1},x_{3}] =x_{5}, [x_{2},x_{4}] =x_{5}$\\ \\
$ L_{5,8} $ &$[x_{1},x_{2}] =x_{4}, [x_{1},x_{3}] =x_{5}$\\  \\
$ L_{6,3} $& $[x_{1},x_{2}] =x_{3}, [x_{1},x_{3}] =x_{4}$\\ \\
$ L_{6,5} $& $[x_{1},x_{2}] =x_{3}, [x_{1},x_{3}] =x_{5}, [x_{2},x_{4}] =x_{5}$\\ \\
$ L_{6,8} $ &$[x_{1},x_{2}] =x_{4}, [x_{1},x_{3}] =x_{5}$\\  \\
 $ L_{6,10} $ & $[x_{1},x_{2}] =x_{3}, [x_1, x_3]=x_6, [x_{4},x_{5}] =x_{6}$ \\ \\
 $ L_{6,22}(\varepsilon) $& $[x_{1},x_{2}] =x_{5}, [x_{1},x_{3}] =x_{6}, [x_{2},x_{4}]  =\varepsilon  x_{6}, [x_3, x_4]=x_5$ \\ \\
\hline
\label{ta1}
\end{longtable}
\begin{longtable}{cccc}
\caption{: $ 7 $-dimensional nilpotent Lie algebras  with $  \dim L^{2}=2$}\\
\hline \multicolumn{1}{c}{\textbf{Name}} & \multicolumn{1}{c}{\textsf{Nonzero multiplication}}    \\
\hline
\endhead
\hline \multicolumn{2}{r}{\small \itshape continued on the next page}
\endfoot
\endlastfoot
$ L_{6,3}\oplus A(1) $& $[x_{1},x_{2}] =x_{3}, [x_{1},x_{3}] =x_{4}$\\ \\
$ L_{6,5}\oplus A(1) $& $[x_{1},x_{2}] =x_{3}, [x_{1},x_{3}] =x_{5}, [x_{2},x_{4}] =x_{5}$\\ \\
$ L_{6,8}\oplus A(1) $ &$[x_{1},x_{2}] =x_{4}, [x_{1},x_{3}] =x_{5}$\\  \\
 $ L_{6,22}(\varepsilon)\oplus A(1) $& $[x_{1},x_{2}] =x_{5}, [x_{1},x_{3}] =x_{6}, [x_{2},x_{4}]  =\varepsilon  x_{6}, [x_3, x_4]=x_5$ \\ \\
 $ L_{6,10}\oplus A(1) $ & $[x_{1},x_{2}] =x_{3}, [x_1, x_3]=x_6, [x_{4},x_{5}] =x_{6}$ \\ \\
 $ 27A$& $[x_{1},x_{2}] =x_{6}, [x_{1},x_{4}] =x_{7}, [x_{3},x_{5}]  =x_{7}$\\ \\
$ 27B$&$[x_{1},x_{2}]= [x_{3},x_{4}] =x_{6}, [x_{1},x_{5}] =[x_{2},x_{3}] =x_{7}$\\ \\
$ 157 $& $[x_{1},x_{2}] =x_{3}, [x_{1},x_{3}] =x_{7}, [x_{2},x_{4}] =x_{7}, [x_5, x_6]=x_7$\\

\hline
\label{ta2}
\end{longtable}
 \begin{longtable}{cccc}
\caption{: nilpotent Lie algebras of dimension at most $ 6 $ with $  \dim L^{2}=3$}\\
\hline \multicolumn{1}{c}{\textbf{Name}} & \multicolumn{1}{c}{\textsf{Nonzero multiplication}}    \\
\hline
\endhead
\hline \multicolumn{2}{r}{\small \itshape continued on the next page}
\endfoot
\endlastfoot
$ L_{5,6} $ &$[x_{1},x_{2}] =x_{3}, [x_{1},x_{3}] =x_{4}, [x_{1},x_{4}] =[x_{2},x_{3}] =x_{5}$\\  \\
 $ L_{5,7} $& $[x_{1},x_{2}] =x_{3}, [x_{1},x_{3}] =x_{4}, [x_{1},x_{4}]  =x_{5}$ \\ \\
$ L_{5,9} $& $[x_{1},x_{2}] =x_{3}, [x_{1},x_{3}] =x_{4}, [x_{2},x_{3}] =x_{5}$ \\ \\
$ L_{6,6}$&$[x_{1},x_{2}] =x_{3}, [x_{1},x_{3}] =x_{4}, [x_{1},x_{4}] =[x_{2},x_{3}] =x_{5}$\\ \\
$ L_{6,7} $& $[x_{1},x_{2}] =x_{3}, [x_{1},x_{3}] =x_{4}, [x_{1},x_{4}]  =x_{5}$\\ \\
$ L_{6,9} $& $[x_{1},x_{2}] =x_{3}, [x_{1},x_{3}] =x_{4}, [x_{2},x_{3}] =x_{5}$\\ \\
$ L_{6,11}$ &$[x_{1},x_{2}] =x_{3}, [x_{1},x_{3}] =x_{4},$ \\ 
&$[x_{1},x_{4}] =[x_{2},x_{3}] =[x_{2},x_{5}]=x_{6}$ \\ 
\\
 $ L_{6,12} $& $[x_{1},x_{2}] =x_{3}, [x_{1},x_{3}] =x_{4}, [x_{1},x_{4}] =[x_{2},x_{5}] =x_{6}$  \\
  \\
   $ L_{6,13} $& $[x_{1},x_{2}] =x_{3}, [x_{1},x_{3}] = [x_{2},x_{4}] =x_{5},$ \\&$[x_{1},x_{5}] = [x_{3},x_{4}] =x_{6}$ \\
 \\
 $ L_{6,19}(\epsilon) $ & $[x_{1},x_{2}] =x_{4}, [x_{1},x_{3}] =  x_{5}, [x_{1},x_{5}] =[x_{2},x_{4}] =x_{6}, $\\ 
 &$[x_{3},x_{5}] =\epsilon x_{6}$   \\
 \\
$ L_{6,20} $ & $[x_{1},x_{2}] =x_{4}, [x_{1},x_{3}]  =x_{5},[x_{1},x_{5}] = [x_{2},x_{4}] =x_{6}$ \\ \\
 $ L_{6,23} $&$[x_{1},x_{2}] =x_{3}, [x_{1},x_{3}] =[x_{2},x_{4}] =x_{5}, [x_{1},x_{4}] =x_{6}$\\
 \\
 $ L_{6,24}(\epsilon) $ &$[x_{1},x_{2}] =x_{3}, [x_{1},x_{3}] =[x_{2},x_{4}] =x_{5},$  \\
 & $[x_{1},x_{4}] =\varepsilon x_{6}, [x_{2},x_{3}] =x_{6}$\\
 \\
 $ L_{6,25} $& $[x_{1},x_{2}] =x_{3}, [x_{1},x_{3}] =x_{5}, [x_{1},x_{4}] =x_{6}$   \\
 \\
$ L_{6,26} $& $[x_{1},x_{2}] =x_{4}, [x_{1},x_{3}] =x_{5}, [x_{2},x_{3}] =x_{6}$   \\

 \hline
 \label{ta3}
\end{longtable}
\begin{longtable}{cccc}
\caption{: $ 7 $-dimensional nilpotent Lie algebras  with $  \dim L^{2}=3$}\\
\hline \multicolumn{1}{c}{\textbf{Name}} & \multicolumn{1}{c}{\textsf{Nonzero multiplication}}    \\
\hline
\endhead
\hline \multicolumn{2}{r}{\small \itshape continued on the next page}
\endfoot
\endlastfoot
$ 37A $& $[x_{1},x_{2}] =x_{5}, [x_{2},x_{3}] =x_{6}, [x_{2},x_{4}] =x_{7}$  \\
\\
 $ 37B $  & $[x_{1},x_{2}] =x_{5}, [x_{2},x_{3}] =x_{6}, [x_{3},x_{4}] =x_{7}$  \\
\\
$ 37C $  & $[x_{1},x_{2}] =[x_{3},x_{4}]=x_{5}, [x_{2},x_{3}] =x_{6}, [x_{2},x_{4}] =x_{7}$   \\
\\
$ 37D $  & $[x_{1},x_{2}] =[x_{3},x_{4}]=x_{5}, [x_{1},x_{3}] =x_{6}, [x_{2},x_{4}] =x_{7}$   \\
\\
$ 257A $&$[x_{1},x_{2}] =x_{3}, [x_{1},x_{3}] =[x_{2},x_{4}]=x_{6}, [x_{1}, x_{5}]=x_{7}$\\
\\
$ 257B $&$[x_{1},x_{2}] =x_{3}, [x_{1},x_{3}] =x_{6}, [x_{1}, x_{4}]=[x_{2},x_{5}]=x_{7}$  \\
\\
$ 257C $ & $[x_{1},x_{2}] =x_{3}, [x_{1},x_{3}] =[x_{2}, x_{4}]=x_{6}, [x_{2},x_{5}]=x_{7}$  \\
\\
$ 257D $&$[x_{1},x_{2}] =x_{3}, [x_{1},x_{3}] =[x_{2}, x_{4}]=x_{6},$ 
\\
&$[x_{1},x_{4}]=[x_{2},x_{5}]=x_{7}$  \\
\\
$ 257E $& $[x_{1},x_{2}] =x_{3}, [x_{1},x_{3}] =[x_{4}, x_{5}]=x_{6}, [x_{2},x_{4}]=x_{7}$ \\
\\
$ 257F $& $[x_{1},x_{2}] =x_{3}, [x_{2},x_{3}] =[x_{4}, x_{5}]=x_{6}, [x_{2},x_{4}]=x_{7}$ \\
\\
$ 257G $& $[x_{1},x_{2}] =x_{3}, [x_{1},x_{3}] =[x_{4}, x_{5}]=x_{6},$\\ 
&$[x_{1},x_{5}]=[x_{2},x_{4}]=x_{7}$ \\
\\
$ 257H $& $[x_{1},x_{2}] =x_{3}, [x_{1},x_{3}] =[x_{2}, x_{4}]=x_{6}, [x_{4},x_{5}]=x_{7}$\\
\\
$ 257I $&$[x_{1},x_{2}] =x_{3}, [x_{1},x_{3}] =[x_{1}, x_{4}]=x_{6},$\\ 
&$[x_{1},x_{5}]=[x_{2},x_{3}]=x_{7}$ \\
\\
$ 257J $&$[x_{1},x_{2}] =x_{3}, [x_{1},x_{3}] =[x_{2}, x_{4}]=x_{6},$
$[x_{1},x_{5}]=[x_{2},x_{3}]=x_{7}$\\
\\
$ 257K $&$[x_{1},x_{2}] =x_{3}, [x_{1},x_{3}] =x_{6}, [x_{2},x_{3}] =[x_{4},x_{5}]=x_{7}$\\
\\
$ 257L $&$[x_{1},x_{2}] =x_{3}, [x_{1},x_{3}] =[x_{2}, x_{4}]=x_{6},$\\ 
&$[x_{2},x_{3}]=[x_{4},x_{5}]=x_{7}$\\
\\
 $ 147A $&  $[x_{1},x_{2}] =x_{4},[x_{1},x_{3}]  =x_{5}, $ \\&$[x_{1},x_{6}] =[x_{2},x_{5}]=  [x_{3},x_{4}] =x_{7}$\\
 \\
  $ 147B $&  $[x_{1},x_{2}] =x_{4},[x_{1},x_{3}]  =x_{5}, $\\&$[x_{1},x_{4}] =[x_{2},x_{6}]=  [x_{3},x_{5}] =x_{7}$\\
 \\
  $ 1457A $&  $[x_{1},x_{i}] =x_{i+1}~~ i=2,3 , ~~[x_{1},x_{4}] =[x_{5},x_{6}]=x_{7}$ \\
 \\
 $ 1457B $ &  $[x_{1},x_{i}] =x_{i+1}~~ i=2,3 , $ \\&$[x_{1},x_{4}] =[x_{2},x_{3}]=[x_{5},x_{6}]=x_{7}$\\
 \\
 \\
$ 137A $ & $[x_{1},x_{2}] =x_{5}, [x_{1},x_{5}] =[x_{3},x_{6}] = x_{7}, [x_{3},x_{4}] =x_{6}$   \\
 \\
$ 137B $ & $[x_{1},x_{2}] =x_{5}, [x_{3},x_{4}] =x_{6},$  
  \\& $[x_{1},x_{5}] =[x_{2},x_{4}] = [x_{3},x_{6}] =x_{7}$ \\
 \\
$ 137C $  & $[x_{1},x_{2}] =x_{5}, [x_{1},x_{4}] =[x_{2},x_{3}] =x_{6},$  
  \\& $[x_{1},x_{6}] =x_{7},  [x_{3},x_{5}] =-x_{7}$ \\
 \\
 $ 137D $ & $[x_{1},x_{2}] =x_{5}, [x_{1},x_{4}] =[x_{2},x_{3}] =x_{6},$
  \\& $[x_{1},x_{6}] =[x_{2},x_{4}] =x_{7},  [x_{3},x_{5}] =-x_{7}$ \\
 \\
$ 1357A $& $[x_{1},x_{2}] =x_{4}, [x_{1},x_{4}] =[x_{2},x_{3}] =x_{5},$  
  \\& $[x_{1},x_{5}] =[x_{2},x_{6}] =x_{7},  [x_{3},x_{4}] =-x_{7}$ \\
 \\
 $ 1357B $ & $[x_{1},x_{2}] =x_{4}, [x_{1},x_{4}] =[x_{2},x_{3}] =x_{5},$  
  \\&$[x_{1},x_{5}] =[x_{3},x_{6}] =x_{7},  [x_{3},x_{4}] =-x_{7}$ \\
 \\
 $ 1357C $ & $[x_{1},x_{2}] =x_{4}, [x_{1},x_{4}] =[x_{2},x_{3}] =x_{5},$ 
  \\& $[x_{1},x_{5}] =[x_{2},x_{4}] =x_{7},  [x_{3},x_{4}] =-x_{7}$ \\
 \hline
 \label{ta4}
\end{longtable}
 \begin{longtable}{ccccc| |cccccc}
 \caption{: $ 7 $-dimensional decomposable nilpotent Lie algebras with $  \dim L^{2}=3$}\\
\hline \multicolumn{2}{c}{\textbf{Name}} & \multicolumn{2}{c}{\textbf{Name}}   \\
\hline
\endhead
\hline \multicolumn{3}{r}{\small \itshape continued on the next page}
\endfoot
\endlastfoot
&$L_{4,3} \oplus H(1)$ & $ L_{6,19}(\epsilon) \oplus A(1)$ \\
&$ L_{5,6} \oplus A(2)$ &$ L_{6,20} \oplus A(1)$\\
&$ L_{5,7} \oplus A(2)$&$ L_{6,23} \oplus A(1)$\\
&$ L_{5,9} \oplus A(2)$&$ L_{6,24}(\epsilon)\oplus A(1)$\\
&$ L_{6,11} \oplus A(1)$& $ L_{6,25} \oplus A(1)$\\
&$ L_{6,12} \oplus A(1)$&$ L_{6,26} \oplus A(1)$\\
&$ L_{6,13} \oplus A(1)$ \\
 \hline
 \label{ta5}
\end{longtable}
\begin{longtable}{cccc}
\caption{: nilpotent Lie algebras of dimension at most $ 6 $ with $  \dim L^{2}=4$}\\
\hline
 \multicolumn{1}{c}{\textbf{Name}} & \multicolumn{1}{c}{\textsf{Nonzero multiplication}}    \\
\hline
\endhead
\hline \multicolumn{2}{r}{\small \itshape continued on the next page}
\endfoot
\endlastfoot
$ L_{6,14} $& $[x_{1},x_{2}] =x_{3}, [x_{1},x_{3}] =x_{4},  [x_{1},x_{4}] =x_{5},$ \\ 
& $ [x_{2},x_{3}] =x_{5}, [x_{2},x_{5}] =x_{6}, [x_{3},x_{4}] =-x_{6},$\\\\
$ L_{6,15} $& $[x_{1},x_{2}] =x_{3}, [x_{1},x_{3}] =x_{4}, [x_{1},x_{4}] =x_{5}, $\\ 
& $ [x_{2},x_{5}] =x_{6}, [x_{3},x_{4}] =-x_{6},  $\\\\
$ L_{6,16} $& $[x_{1},x_{2}] =x_{3}, [x_{1},x_{3}] =x_{4}, [x_{1},x_{4}] =x_{5}$\\ 
& $ [x_{2},x_{5}] =x_{6},  $ $ [x_{3},x_{4}] =-x_{6},  $\\\\
$ L_{6,17} $ &$[x_{1},x_{2}] =x_{3}, [x_{1},x_{3}] =x_{4},$ $ [x_{1},x_{4}] =x_{5},  $\\  
& $ [x_{1},x_{3}] =x_{4}, [x_{1},x_{5}] =x_{6},  $ $ [x_{2},x_{3}] =x_{6},  $\\ \\
$ L_{6,18} $& $[x_{1},x_{2}] =x_{3}, [x_{1},x_{3}] =x_{4},$ $ [x_{1},x_{4}] =x_{5},  $\\ 
&$ [x_{1},x_{5}] =x_{6},  $\\ \\
$ L_{6,21}(\varepsilon) $& $[x_{1},x_{2}] =x_{3}, [x_{1},x_{3}] =x_{4}, [x_{2},x_{3}] =x_{5}$\\ 
& $ [x_{1},x_{4}] =x_{6},  $ $ [x_{1},x_{3}] =x_{4},  $ $ [x_{2},x_{5}] =\varepsilon x_{6}. $\\
\hline
\label{ta61}
\end{longtable}
By using  \cite[Theorem 3.1]{N}, \cite[Theorems 3.9 and 4.5]{s1}, \cite[Main Theorem]{s3}, \cite[Main Theorem]{s4} and  \cite[Theorem 2.3]{s5} the structures of all nilpotent Lie algebras $ L $ with $ s(L)\in \lbrace 0,1,2,3,4,5\rbrace $ are summarized as following.
\begin{thm}\label{s}
Let $ L $ be an $ n $-dimensional nilpotent Lie algebra. 
\begin{itemize}
\item[(i).] If $ s(L)=0, $ then $ L $ is isomorphic to $ H(1)\oplus A(n-3). $
\item[(ii).]If $ s(L)=1, $ then $ L $ is isomorphic to  $ L_{5,8}. $
\item[(iii).]If $ s(L)=2, $ then $ L $ is isomorphic to one of the Lie algebras  $ L_{5,8}\oplus A(1), $ $ L_{4,3} $ or $ H(m)\oplus A(n-2m-1) $ for all $ m\geq 2. $
\item[(iv).]If $ s(L)=3, $ then $ L $ is isomorphic to one of the Lie algebras $ L_{5,8}\oplus A(2), $ $ L_{4,3}\oplus A(1), $ $ L_{5,5}, $  $ L_{6,22}(\varepsilon) $ for all $\varepsilon\in \mathbb{F}/(\stackrel{*}{\sim}) $ or $ L_{6,26}. $  
\item[(v).]If $ s(L)=4, $ then $ L $ is isomorphic to one of the Lie algebras $ L_{5,8}\oplus A(3), $ $ L_{4,3}\oplus A(2), $ $ L_{5,5}\oplus A(1), $  $ L_{6,22}(\varepsilon) \oplus A(1)$ for all $\varepsilon\in \mathbb{F}/(\stackrel{*}{\sim}), $ $ L_{5,6}, $ $ L_{5,7}, $ $ L_{5,9} $ or $ 37A. $  
\item[(vi).]If $ s(L)=5, $ then $ L $ is isomorphic to one of the Lie algebras  $L_{5,8}\oplus A(4), $ $ L_{4,3}\oplus A(3), $ $ L_{5,5}\oplus A(2), $  $ L_{6,22}(\varepsilon) \oplus A(2)$ for all $\varepsilon\in \mathbb{F}/(\stackrel{*}{\sim}), $ $ L_{6,26}\oplus A(1), $ $ L_{6,10}, $ $ L_{6,23}, $ $ L_{6,25}, $ $ 37B, $ $ 37C $ or $ 37D. $ 
\end{itemize}
\end{thm}
The following lemma and propositions are used in the next sections.
 \begin{lem}\label{11**}
 Let $ L $ be a nilpotent Lie algebra  of dimension at most $ 6$ with $ \dim L^{2}=3. $ Then $ s(L), $ the structure and  the Schur multiplier of $L$  are given in the following table.
 \begin{longtable}{ccc}
\caption{}\\
\hline \multicolumn{1}{c}{\textsf{Name}} &\multicolumn{1}{c}{\textbf{$\dim \mathcal{M}(L)$}} & \multicolumn{1}{c}{\textsf{$s(L)$}}      \\
\hline
\endhead
\hline \multicolumn{3}{r}{\small \itshape continued on the next page}
\endfoot
\endlastfoot
$ L_{5,6}, $ $ L_{5,7}, $ $ L_{5,9} $ & $ 3 $  & $ 4 $ \\
\\
$ L_{6,6}, $ $ L_{6,7}, $  $  L_{6,9},  $ $ L_{6,11},  $ $ L_{6,12},  $  $ L_{6,19}(\varepsilon), $ $ L_{6,20},$ $ L_{6,24}(\varepsilon) $  &$ 5 $ & $ 6 $\\
\\
$ L_{6,13}  $  & $ 4 $ & $ 7$ \\
\\
$  L_{6,23},$ $ L_{6,25} $  & $ 6 $ & $ 5 $ \\
\\
$ L_{6,26} $ & $ 8 $ & $ 3 $\\
 \hline
  \label{ta6}
\end{longtable} 
\begin{proof}
The proof is obtained by a modification of the proof \cite[Lemma 2.5]{s4}. 
\end{proof}
 \end{lem}
 \begin{prop}
Let $ L $ be an $ 7 $-dimensional nilpotent Lie algebra with $ \dim L^{2}=3. $ Then the structure,  the Schur multiplier of $ L $ and $ s(L) $ are presented in the following table.
\begin{longtable}{ccc}
\caption{}\\
\hline \multicolumn{1}{c}{\textsf{Name}} &\multicolumn{1}{c}{\textbf{$\dim \mathcal{M}(L)$}} & \multicolumn{1}{c}{\textsf{$s(L)$}}      \\
\hline
\endhead
\hline \multicolumn{3}{r}{\small \itshape continued on the next page}
\endfoot
\endlastfoot
$(37A)$ & $ 12$& $ 4$  \\
\\
$(37B),$  $(37C),$  $(37D)$ & $ 11$& $ 5$\\
\\
$(257A),$  $ 257C, $ $ 257F $ & $ 9$& $ 7$ \\
\\
$(257B),$ $ 257D, $ $ 257E, $ $(257G),$ $(257H),$ $(257I),$ $(257J)$& $8 $& $ 8$ \\
\\
$ 147A, $ $ 147B, $ $L_{4,3}\oplus H(1),$  $L_{5,6}\oplus A(2),$  $L_{5,7}\oplus A(2),$  $L_{5,9}\oplus A(2)$ & $8 $& $ 8$\\
\\
 $L_{6, 11}\oplus A(1),$  $L_{6,12}\oplus A(1),$ $L_{6,19}(\varepsilon)\oplus A(1),$ $L_{6,20}\oplus A(1)$ & $8 $& $ 8$ \\
 \\
 $L_{6,24}(\varepsilon)\oplus A(1),$ $(257K),$ $(257L),$ & $8 $& $ 8$\\
 \\
 $ 1457A, $ $ 1457B ,$ $(1357B),$ $(1357C)$& $ 6$& $ 10$ \\
 \\
$(137A),$ $(137B),$ $(137C),$ $(137D),$ $(1357A),$ $L_{6,13}\oplus A(1)$& $7 $& $9 $\\
\\
$L_{6,23}\oplus A(1),$  $ L_{6,25}\oplus A(1) $& $ 9$& $ 7$\\
\\
$L_{6,26}\oplus A(1)$ & $ 11$& $ 5$\\

 \hline
  \label{ta7}
\end{longtable} 
\end{prop}
\begin{proof}
The proof is  obtained similar to the proof of   \cite[Lemma 2.5]{s4}. 
\end{proof}
\begin{prop}
Let $ L $ be a nilpotent Lie algebra of dimension at most $ 6 $ with $ \dim L^{2}=4. $ Then the structure,  the Schur multiplier of $ L $ and $ s(L) $ are presented in the following table.
\begin{longtable}{ccc}
\caption{}\\
\hline \multicolumn{1}{c}{\textsf{Name}} &\multicolumn{1}{c}{\textbf{$\dim \mathcal{M}(L)$}} & \multicolumn{1}{c}{\textsf{$s(L)$}}      \\
\hline
\endhead
\hline \multicolumn{3}{r}{\small \itshape continued on the next page}
\endfoot
\endlastfoot
$ L_{6,14}, $  $ L_{6,16} $ & $ 2 $  & $ 9 $ \\
\\
$ L_{6,15}, $ $ L_{6,17},  $ $ L_{6,18}$ &$ 3$ & $ 8 $ \\
\\
$ L_{6,21}(\varepsilon)  $ & $ 4 $ & $ 7 $ \\
\\
\hline
  \label{ta81}
\end{longtable} 
\end{prop}
\begin{proof}
The proof is obtained by a similar argument as in   the proof of \cite[Lemma 2.5]{s4}. 
\end{proof}
\section{the structures of a nilpotent Lie algebra  $ L $  with $ s(L)=6 $}
In this section, we determine the structures of a all nilpotent Lie algebras $ L $ when $ s(L)=6. $ First, we begin by the following propositions which are needed in the proof of the next theorems. 
\begin{prop}\label{pr2.1}
There is no  $ n $-dimensional nilpotent Lie algebra with $ s(L)=6, $ when 
\begin{itemize}
\item[(i).] $ \dim L^{2}\geq 5, $
\item[(ii).] $  \dim L^{2}=1. $
\end{itemize}
\end{prop}
\begin{proof}
The proof is obtained in the same way as the proof  \cite[Proposition 2.4]{s4}.
\end{proof}
\begin{prop}\label{pr2.3}
Let $ L $ be a non-capable nilpotent Lie algebra of dimension at most $ 7 $ with $ \dim L^{2}=2. $ Then $ s(L), $ the structure and   the Schur multiplier of $ L $   are given in the following table.
\begin{longtable}{ccc}
\caption{}\\
\hline \multicolumn{1}{c}{\textsf{Name}} &\multicolumn{1}{c}{\textbf{$\dim \mathcal{M}(L)$}} & \multicolumn{1}{c}{\textsf{$s(L)$}}      \\
\hline
\endhead
\hline \multicolumn{3}{r}{\small \itshape continued on the next page}
\endfoot
\endlastfoot
$ L_{6,10} $ & $ 6 $  & $ 5 $ \\
\\
$ 27A,  $  $ L_{6,10}\oplus A(1), $ $ 157 $ & $ 10 $ & $ 6$ \\
\\
\hline
  \label{ta8}
\end{longtable} 
\end{prop}
\begin{proof}
By using \cite[Theorem 1.1]{P}, $ L_{6,10} $ is only non-capable Lie algebra with derived subalgebra of dimension $ 2 $ such that $ n\leq 6. $ From \cite[Proposition 3.1]{P}, we have $ \dim \mathcal{M}(L_{6,10})=6 $ and so $ s(L_{6,10})=5. $ Also, by using Table \ref{ta2}, we have  $ L_{6,3}\oplus A(1), $ $ L_{6,5}\oplus A(1), $ $ L_{6,8}\oplus A(1), $ $ L_{6,22}(\varepsilon)\oplus A(1) $ and  $ L_{6,10}\oplus A(1) $ are  decomposable $ 7 $-dimensional nilpotent Lie algebra  and only  $ L_{6,10}\oplus A(1) $ is non-capable by using \cite[Proposition 3.1]{J}. From \cite[Theorem 1]{B1}, $ \dim \mathcal{M}(L_{6,10}\oplus A(1))=10 $ and so $ s(L_{6,10}\oplus A(1))=6. $\\
From Table \ref{ta2},  we know that $ 27A, $ $ 27B $ and $ 157 $ are $ 7 $-dimensional Lie algebra  with derived subalgebra of dimension $ 2. $ On the other hand, \cite[Proposition 2.10 and Lemma 2.11]{J1} show  $ 27A $ and $ 27B $ are non-capable and capable, respectively and $ \dim \mathcal{M}(27A)=10 $ and $ \dim \mathcal{M}(27B)=9, $ so  $s(27A)=6$ and $  s(27B)=7. $  By a similar way and using \cite[Proposition 2.10 and Lemma 2.11]{J1},  the Lie algebra $ 157 $ is non-capable and $ \dim \mathcal{M}(157)=10,$ so $ s(157)=6. $
Therefore $ L $ is isomorphic to one of the Lie algebras $ L_{6,10}, $  $ L_{6,10}\oplus A(1), $ $ 27A$ or $157.$
\end{proof}
\begin{thm}\label{th2.4}
Let $ L $ be an $ n $-dimensional nilpotent Lie algebra with $ s(L)=6 $ and $ \dim L^{2}=2. $ Then $ L $ is isomorphic to one of the nilpotent Lie algebras  $L_{5,8} \oplus A(5), L_{4,3}\oplus A(4),$    $L_{5,5}\oplus A(3)$,  $L_{6,22}(\varepsilon)\oplus A(3)$ for all $ \varepsilon \in \mathbb{F}/(\stackrel{*}{\sim}), $  $ L_{6.10}\oplus A(1), $  $ 27A $ or $ 157. $
\end{thm}
\begin{proof}
If $ L $  is a capable Lie algebra, then the proof is similar to \cite[Theorem 2.6]{s4} and $ L $ is isomorphic to one of the nilpotent Lie algebras  $L_{5,8} \oplus A(5), L_{4,3}\oplus A(4)$, $L_{5,5}\oplus A(3)$ or  $L_{6,22}(\varepsilon)\oplus A(3)$ for all $\varepsilon \in  \mathbb{F}/(\stackrel{*}{\sim}). $ 
\\
If $ L $  is a non-capable Lie algebra, then $ n\leq 8 $ by using Lemma \ref{l1.1}. Let $ n\leq 7. $ Then  $ L $ is isomorphic to one of the non-capable Lie algebras  $ 27A, $ $ L_{6,10}\oplus A(1) $ or $ 157 $  by using  Proposition \ref{pr2.3}. Let  now $ n=8. $  Since $ Z^{\wedge}(L)\subseteq L_{2} $ and $ n=8, $ $ L/Z^{\wedge}(L) $ is isomorphic to $ A(6) $ or $ H(1)\oplus A(4). $ On the other hand,  $ L\wedge L\cong  L/Z^{\wedge}(L) \wedge L/Z^{\wedge}(L) $ by using \cite[Corollary 2.3]{N1}. Thus $ \dim L\wedge L $ is equal to $ 15 $ or $ 17 $  and so $ \mathcal{M}(L) $ is isomorphic to  $ A(13) $ or $ A(15) $ by using \cite[Theorem 35 (iii)]{E}. Hence $ s(L)=9 $ or $ s(L)=7, $ respectively. Therefore there is no $ 8 $-dimensional non-capable nilpotent Lie algebra with $ \dim L^{2}=2 $ and $ s(L)=6. $
\end{proof}
Recall that a Lie algebra $ L $ is called generalized Heisenberg of rank $ n $ if $ L^{2}=Z(L) $  and $ \dim L^{2}=n. $ 
\begin{lem}\label{l2.3}
Let $ L $ be an $ n $-dimensional generalized Heisenberg of rank $ 3 $ with $ s(L)=6. $ Then $ n\leq 7. $
\end{lem}
\begin{proof}
The proof is  obtained  by a modification of  the  proof of \cite[Lemma 2.8]{s5}.
\end{proof}
\begin{thm}\label{th2.5}
Let $ L $ be an $ n $-dimensional nilpotent Lie algebra with $ s(L)=6 $ and $ \dim L^{2}=3. $ Then $ L $ is isomorphic to one of the nilpotent Lie algebras  $ 37A\oplus A(1), $ $ L_{6,6}, $ $ L_{6,7}, $ $ L_{6,9}, $ $ L_{6,11}, $ $ L_{6,12}, $ $    L_{6,19}(\varepsilon)$  for all $ \varepsilon \in \mathbb{F^{*}}/(\stackrel{*}{\sim}), $ $ L_{6,20} $ or $ L_{6,24}(\varepsilon)$ for all $ \varepsilon \in \mathbb{F}/(\stackrel{*}{\sim}). $
\end{thm}
\begin{proof}
If $ Z(L)\not\subseteq L^{2}, $ then there is a central one dimensional ideal $ I $ of $ L $ such that $ I\cap L^{2}=0. $ Since $ \dim (L/I)^{2}=3, $ by using  Theorem \ref{th1.1}, we have 
\begin{align*}
\dim \mathcal{M}(L/I)\leq \frac{1}{2}n(n-5)+1.
\end{align*}
If the equality holds, then
\begin{equation*}
\frac{1}{2}(n-2)(n-3)+1-s(L/I)=\dim \mathcal{M}(L/I)=\frac{1}{2}n(n-5)+1
\end{equation*}
and so $ s(L/I)=3. $ By using Table \ref{ta6}, $ L/I \cong L_{6,26}. $ On the other hand,  $ Z(L)\not\subseteq L^{2} $ and $ I \not\subset L^{2}$ thus $ L\cong L_{6,26}\oplus A(1). $ Hence    $ s( L_{6,26}\oplus A(1))=5$ by using Table \ref{ta7}.  It is a contradiction. Thus
\begin{align*}
\dim \mathcal{M}(L/I)\leq \frac{1}{2}n(n-5).
\end{align*}
If the equality holds, then
\begin{align*}
\frac{1}{2}(n-2)(n-3)+1-s(L/I)=\dim \mathcal{M}(L/I)=\frac{1}{2}n(n-5)
\end{align*}
and $ s(L/I)=4. $ Then $ L/I $ is isomorphic to one of the Lie algebras $ L_{5,6}, $ $ L_{5,7}, $ $ L_{5,9} $ or $ 37A. $ Since  $ Z(L)\not\subseteq L^{2} $  and  $ I\not\subset L^{2}, $ we have  $ L $ is isomorphic to  one of the Lie algebras $ L_{6,6}\cong L_{5,6}\oplus A(1), $ $L_{6,7}\cong L_{5,7}\oplus A(1), $ $ L_{6,9}\cong L_{5,9}\oplus A(1) $ or $ 37A\oplus A(1). $ By using Table \ref{ta6}, we have $s(L_{6,6})=s(L_{6,7})=s(L_{6,9})=6 $ and  looking at Table \ref{ta7} and \cite[Theorem 1]{B1} we have $ s(37A\oplus A(1))=6. $ \\
If $ \dim \mathcal{M}(L/I)\leq \frac{1}{2}n(n-5)-1,  $ then Theorem \ref{th1.1} implies
\begin{equation*}
\dim \mathcal{M}(L)=\frac{1}{2}(n-1)(n-2)-5\leq \frac{1}{2}n(n-5)-1+n-4=\frac{1}{2}(n^{2}-3n)-5
\end{equation*}
which is a contradiction. \\
If $ Z(L)\subseteq L^{2}, $ then there is a one dimensional ideal $ I $ of $ L $ such that $ I\subseteq Z(L)\cap L^{2}. $ Hence by using Theorem \ref{th1.2} we have
\begin{align*}
\frac{1}{2}(n-1)(n-2)-4\leq \frac{1}{2}(n-2)(n-3)+1-s(L/I)+n-3
\end{align*}
and so $ s(L/I)\leq 4. $ Since $ \dim (L/I)^{2}=2, $ by using  Theorem \ref{s} $ L/I$ is isomorphic to one of the Lie algebras  $  L_{5,8}, $ $ L_{4,3}, $ $ L_{5,8}\oplus A(1), $ $ L_{4,3}\oplus A(1), $ $ L_{5,8}\oplus A(2), $  $ L_{5,5}, $ $ L_{6,22}(\varepsilon) $ for all $ \varepsilon \in \mathbb{F}/(\stackrel{*}{\sim}), $   $ L_{4,3}\oplus A(2), $ $ L_{5,8}\oplus A(3), $ $ L_{5,5}\oplus A(1)$ or $ L_{6,22}(\varepsilon)\oplus A(1) $ for all $ \varepsilon \in \mathbb{F}/(\stackrel{*}{\sim}). $  Hence $ n\leq 9. $  \\
If $ n\leq 7, $ then $ L $ is isomorphic to $ L_{6,6}, $ $ L_{6,7}, $ $ L_{6,9}, $ $ L_{6,11}, $ $ L_{6,12}, $ $    L_{6,19}(\varepsilon)$ for all $ \varepsilon \in \mathbb{F^{*}}/(\stackrel{*}{\sim}), $   $ L_{6,20} $ or $ L_{6,24}(\varepsilon) $ for all $ \varepsilon \in \mathbb{F}/(\stackrel{*}{\sim}) $ by looking at Tables \ref{ta6}, \ref{ta7}. \\
Let $ L/I\cong   L_{5,8}\oplus A(2),$ then $ n=8. $ If $ L $ is of nilpotency class $ 2, $ then there is no such a Lie algebra by using Lemma \ref{l2.3}. 
 Let $ L $ is of nilpotency class $ 3 $. Then $ \dim \gamma_{3}(L)=1. $ On the other hand,   $ \mathcal{M}(L_{5,8})=9 $ by using \cite[Lemma 2.10]{P} and $ \mathcal{M}(A(2))=1 $ via  \cite[Lemma 23]{M}. Hence   $ \dim \mathcal{ M}(L/I)=13 $ by invoking \cite[Theorem 1]{B1}. Using \cite[proof of Theorem 1.1]{R}, we have 
  \begin{equation}\label{1}
2=n-m-c \leq\dim \ker \lambda_{3}, ~~\text{and} 
\end{equation}
\begin{equation}\label{2}
\dim \mathcal{M}(L)=\dim \mathcal{M}(L/\gamma_{3}(L))-\dim \gamma_{3}(L)+\dim L/\gamma_{2}(L)\otimes \gamma_{3}(L)-\dim \ker \lambda_{3},
\end{equation}
so
\begin{align*}
16=\dim \mathcal{M}(L) & \leq \dim \mathcal{M}(L/\gamma_{3}(L))-\dim \gamma_{3}(L)+\dim L/\gamma_{2}(L)\otimes \gamma_{3}(L)-2\cr
& \leq 15.
\end{align*}
It is a contradiction. \\
If $ L/I\cong   L_{6,22}(\varepsilon)\oplus A(1),$ then $ n=8$ and  we get  similarly contradiction.   \\
If $ L/I\cong   L_{5,8}\oplus A(3),$ then $ n=9 $ and we have a contradiction. The proof is completed.
\end{proof}
From \cite{M}, a Lie algebra $ L $ is called a stem Lie algebra if   $ Z(L)\subseteq L^{2}. $
\begin{lem}\label{l2.8}
There is no stem $ 7 $-dimensional Lie algebra with $ s(L)=6 $ or $ 7 $  such that $ L/I\cong L_{6,26}, $ $ \dim I=1 $  and $ I\subseteq Z(L). $
\end{lem}
\begin{proof}
Let $ L $ be a nilpoten Lie algebra of class two. By contrary, let there is a stem $ 7 $-dimensional Lie algebra $ L $ of nilpotency class two such that $ L/I\cong L_{6,26} $ and $ I\subseteq Z(L). $ 
Let $ I=\langle x_7 \rangle. $ Then
\begin{align*}
&[x_1, x_2]=x_4+\alpha_1 x_7, &&  [x_1, x_3]=x_5+\alpha_2 x_7, && [x_1, x_4]=\alpha_3 x_7, \cr
&[x_1, x_5]=\alpha_4 x_7,  &&[x_1, x_6]=\alpha_5 x_7, && [x_2, x_3]=x_6+\alpha_6 x_7, \cr
&[x_2, x_4]=\alpha_7 x_7,  &&[x_2, x_5]=\alpha_8 x_7, && [x_2, x_6]=\alpha_9 x_7, \cr
&[x_3, x_4]=\alpha_{10} x_7,  &&[x_3, x_5]=\alpha_{11} x_7, && [x_3, x_6]=\alpha_{12} x_7, \cr
&[x_4, x_5]=\alpha_{13} x_7,  &&[x_4, x_6]=\alpha_{14} x_7, &&  [x_5, x_6]=\alpha_{15} x_7. 
\end{align*}
Since $ L $ is of nilpotency class two, a   change of variables $ x'_4=x_4+\alpha_1 x_7, $ $x'_5= x_5+\alpha_2 x_7 $ and $ x'_6=x_6+\alpha_6 x_7 $ and relabeling, we have 
\begin{equation*}
L=\langle x_1, \dots, x_7\mid [x_1, x_2]=x_4,  [x_1, x_3]=x_5, [x_2, x_3]=x_6 \rangle \cong L_{6,26}\oplus A(1).
\end{equation*}
On the other hand, $ L $ is stem, thus it is a contradiction.\\
Let $ L $ be of nilotency class $ 3. $ Since  $ L_{6,26}  $ is of nilpotency class two and $ L/I\cong L_{6,26},  $ we have $ \dim L^{3}=1. $ Now, by looking at the classification all $ 7 $-dimensional nilpotent Lie algebras in \cite{Mi},  $ L $ should be isomorphic to one of the Lie algebras 
\[357A=\langle x_1, \dots, x_7\mid  [x_{1},x_{2}] =x_{3}, [x_{1},x_{3}] =x_{5}, [x_{1},x_{4}] = x_{7}, [x_{2},x_{4}] =x_{6} \rangle, \]
\[247N=\langle    x_1, \dots, x_7\mid [x_{1},x_{i}] =x_{i+2},~~i=2,3, [x_{1},x_{5}]=x_{6},   [x_{2},x_{3}]=x_{7},  [x_{2},x_{4}] =x_{6}\rangle, \]
\[ 147F=\langle x_1, \dots, x_7\mid  [x_{1},x_{2}] =x_{4}, [x_{1},x_{3}] =-x_{6}, [x_{1},x_{5}] =x_{7}, [x_{1},x_{6}] =x_{7}\]
\[ [x_{2},x_{3}] =x_{5}, [x_{2},x_{4}] =x_{7}, [x_{2},x_{6}] =x_{7}, [x_{3},x_{4}] =x_{7} \rangle. \]
\[147D=\langle  x_1, \dots, x_7\mid [x_{1},x_{2}] =x_{4}, [x_{1},x_{3}] =-x_{6}, [x_{1},x_{5}] = x_{7}, [x_{1},x_{6}] =x_{7}\]
\[ [x_{2},x_{3}] =x_{5}, [x_{2},x_{6}] =x_{7}, [x_{3},x_{4}] = -2x_{7} \rangle,\]
\[\text{or}~~147E=\langle x_1, \dots, x_7\mid [x_{1},x_{2}] =x_{4}, [x_{1},x_{3}] =-x_{6}, [x_{1},x_{5}] =-x_{7}, [x_{2},x_{3}] =x_{5}\]
\[ [x_{2},x_{6}] =\lambda x_{7}, [x_{3},x_{4}] =(1-\lambda)x_{7}\rangle\]
One parameter family, with an invariant  $I(\lambda)=\frac{(1-\lambda+\lambda^{2})^{3}}{\lambda^{2}(\lambda-1)^{2}}, \lambda \neq 0, 1.$ When $ \lambda=0 $ or $ 1, $ it is isomorphic to $ 247P.$\\
By using the method of Hardy and Stitzinger in \cite{B1}, we can see that $ \dim \mathcal{M}(357A)=8 ,$ $ \dim \mathcal{M}(247N)=7 ,$ and $ \dim \mathcal{M}(147D)=\dim \mathcal{M}(147E) =\dim \mathcal{M}(147F)=7,$ so $ s(357A)=14 $ and $ s(247N)=s(147D)=s(147E)=s(147F)=15. $   Therefore there is no  $ 7 $-dimensional Lie algebra  such that $ L/I\cong L_{6,26}, $ $ \dim I=1 $  and $ I\subseteq Z(L). $
\end{proof}
\begin{thm}\label{th2.6}
There is no  $n$-dimensional nilpotent Lie algebra with $ s(L)=6 $ and $ \dim L^{2}=4. $ 
\end{thm}
\begin{proof}
Let $ L $ be a Lie algebra such that $ Z(L)\not\subset L^{2}. $ Then there exist a central one dimension ideal $ I $ of $ L $ such that $ I\cap L^{2}=0. $ Since $ \dim \mathcal{M}(L)=\frac{1}{2}(n-1)(n-2)-5, $ we have
\begin{align*}
\frac{1}{2}(n-1)(n-2)-5\leq \frac{1}{2}(n-2)(n-3)+1-s(L/I)+n-5
\end{align*}
by using Theorem \ref{th1.2}. Hence $ s(L/I)\leq 3. $ On the other hand, $ (L/I)^{2}=4. $ Thus by using Theorem \ref{s} for $ L/I, $  there is no  Lie algebra with $ s(L/I)\leq 3 $ and  $ (L/I)^{2}=4. $ \\
Assume that $ Z(L)\subseteq L^{2}. $ Then   there exists a  one dimension ideal $ I $ of $ L $ such that $ I\subseteq Z(L)\cap L^{2}.$ By using Theorem \ref{th1.2}, we have
\begin{align*}
\frac{1}{2}(n-1)(n-2)-4\leq \frac{1}{2}(n-2)(n-3)+1-s(L/I)+n-4
\end{align*}
and so $ s(L/I)\leq 3. $ Since $ s(L/I)\leq 3 $ and $ \dim(L/I)^{2}=3, $ by using Theorem \ref{s}, we have $ L/I\cong L_{6,26} $ and so $ n=7. $ Thus there is no such a Lie algebra by using Lemma \ref{l2.8}, as required.
\end{proof}
\begin{thm}
Let $ L $ be an $ n $-dimensional nilpotent Lie algebra with $ s(L)=6. $  Then $ L $ is isomorphic to one of the Lie algebras $L_{5,8}\oplus A(5), L_{4,3}\oplus A(4)$, $L_{5,5}\oplus A(3)$,  $L_{6,22}(\varepsilon)\oplus A(3)$ for all $\varepsilon \in  \mathbb{F}/(\stackrel{*}{\sim}), $ $ L_{6.10}\oplus A(1), $  $ 27A, $  $ 157, $ 
 $ 37A\oplus A(1), $ $ L_{6,6}, $ $ L_{6,7}, $ $ L_{6,9}, $ $ L_{6,11}, $ $ L_{6,12}, $ $    L_{6,19}(\varepsilon)$ for all $ \varepsilon \in \mathbb{F^{*}}/(\stackrel{*}{\sim}), $  $ L_{6,20} $ or $ L_{6,24}(\varepsilon) $ for all $ \varepsilon \in \mathbb{F}/(\stackrel{*}{\sim}). $
\end{thm}
\begin{proof}
The proof is obtained  by using Proposition \ref{pr2.1} Theorems \ref{th2.4}, \ref{th2.5} and \ref{th2.6}. 
\end{proof}
\section{ the structures of a nilpotent Lie algebra $ L $  with $ s(L)=7 $}
In this section, we  classify  all  $ n $-dimensional nilpotent Lie algebras $ L $ with $ s(L)=7. $ We begin by the following proposition. 
\begin{prop}\label{pr3.1}
There is no  $ n $-dimensional nilpotent Lie algebra with $ s(L)=7, $ when 
\begin{itemize}
\item[(i).] $ \dim L^{2}\geq 5, $
\item[(ii).] $  \dim L^{2}=1. $
\end{itemize}
\end{prop}
\begin{proof}
The proof is obtained by the same argument as in the proof of  \cite[Proposition 2.4]{s4}.
\end{proof}
\begin{prop}\label{pr3.2}
Let $ L $ be a non-capable  $ 8 $-dimensional Lie algebra with derived subalgebra of dimension $ 2 $ and $ s(L)=7 $ such that $ L/Z^{\wedge}(L)\cong H(1)\oplus A(4). $ Then $ L $ is isomorphic to one of  the nilpotent Lie algebras $ L_{6,10}\oplus A(2), $ $ 27A\oplus A(1), $ 
$ 157\oplus A(1), $  $ L_{6,10}\dotplus H(1),  $ $ H(1)\oplus H(2) $ or
 \[S_1=\langle x_1, \dots, x_8 \mid  [x_1, x_2]=x_6, [x_1, x_4]=x_8, [x_3, x_5]=x_8, [x_{2},x_{7}]= x_8 \rangle.
\]
\end{prop}
\begin{proof}
Let $ L $ be a Lie algebra and $ Z(L)\not\subseteq L^{2}. $  By using \cite[Theorem 1.1]{P} and looking at the classification of all nilpotent Lie algebras in \cite{Mi}, $ L_{6,10}, $ $ 27A $ and $ 157 $ are non-capable Lie algebras with derived subalgebra of dimension $ 2. $ Since $ Z(L)\not\subseteq L^{2}, $
 $ L $ is isomorphic to one of the nilpotent Lie algebras $ L_{6,10}\oplus A(2), $ $ 27A\oplus A(1) $ or  $ 157\oplus A(1)$ by using \cite[Proposition 3.1]{J}. Now, by using Tables \ref{ta8}, \cite[Lemma 23]{M} and \cite[Theorem 1]{B1} we have $ \dim \mathcal{M}(L_{6,10}\oplus A(2))=\dim \mathcal{M}(27A\oplus A(1))=\dim \mathcal{M}(157\oplus A(1))=15, $ so  $ s(L_{6,10}\oplus A(2))=\dim s(27A\oplus A(1))=s(157\oplus A(1))=7. $
\\
Let $ L $ be a Lie algebra and $ Z(L)\subseteq L^{2}. $ Consider the following cases.\\
Case 1.  If $ \dim Z(L)=1, $ then $ L $ is of nilpotency class $ 3 $ and since $ L $ is non-capable, $ L $ is isomorphic to $ L_{6,10}\dotplus H(1) $ by using \cite[Theorem 4.6]{J4}.\\
Case 2. Let $ \dim Z(L)=2. $  Then $ L^{2}=Z(L) $ and $ L $ is of nilpotency class $ 2. $ There exist ideals $ I_1/Z^{\wedge}(L) $ and $ I_2/Z^{\wedge}(L) $ of $ L/Z^{\wedge}(L) $ such that  $ I_1/Z^{\wedge}(L) \cong H(1)\oplus A(3) $ and $ I_2/I\cong A(1). $  Since  $ L $ is of nilpotency class $ 2, $  $ I_{1} $ is of nilpotency class $ 2. $ Since $ L $ is a  nilpotent Lie algebra, we have $ Z^{\wedge}(L)\subseteq L^{2} $ by using \cite[Corollary 2.3]{J4} On the other hand, $ I_1 \cap I_2=Z^{\wedge}(L) \subseteq I_1.$ Hence $ Z^{\wedge}(L)\subseteq I_1 \cap L^{2}=I_1^{2} $ and so $ Z^{\wedge}(L)\subseteq I_1^{2}. $  Since $ L^{2}/ Z^{\wedge}(L)=I_1^{2}+Z^{\wedge}(L)/Z^{\wedge}(L)$ and $ Z^{\wedge}(L)\subseteq I_1^{2}, $ we have $ I_1^{2}=L^{2} $ and $ Z^{\wedge}(L)\subseteq Z^{\wedge}(I_1). $ Thus $ I_1 $ is non-capable Lie algebra.
By using the classification of all nilpotent Lie algebras in \cite{G, Mi} we have $ I_{1} $ is isomorphic to  $ 27A. $  On the other hand, $ I_{2}=A(1)\oplus Z^{\wedge}(L) $ and $ L=I_1+A(1)\oplus Z^{\wedge}(L)=I_1+A(1). $
Since $ L $ is a stem Lie algebra and $ [I_1, I_2] \subseteq Z^{\wedge}(L),$ we have $ [I_1, A(1)]= Z^{\wedge}(L).$ Otherwise, $ L=I_1\oplus A(1) $ and since $ L $ is stem Lie algebra,  it is a contradiction. 
 Now, we are going to obtain $ [27A, A(1)]. $  Let  $ A(1)\cong \langle x_7 \rangle  $ and 
\[27A\cong \langle x_1, \dots, x_6, x_8 \mid [x_1, x_2]=x_6, [x_1, x_4]=x_8, [x_3, x_5]=x_8    \rangle. 
\] 
 We have
\begin{align*}
 &[x_{1},x_{7}]=\alpha_{1}x_8,& & [x_{2},x_{7}]= \alpha_{2}x_8, &  & [x_{3},x_{7}]=\alpha_{3}x_8,   \cr
 &[x_{4},x_{7}]=\alpha_{4}x_8,& &[x_{5},x_{8}]=\alpha_{5}x_8.
\end{align*}
 Put  $ x'_7=x_7-\alpha_3 x_5.$  By relabeling we have
\begin{align*}
 &[x_{1},x_{7}]=\alpha_{1}x_8,& & [x_{2},x_{7}]= \alpha_{2}x_8, &  & [x_{3},x_{7}]=0,   \cr
 &[x_{4},x_{7}]=\alpha_{4}x_8,& &[x_{5},x_{7}]=\alpha_{5}x_8.
\end{align*}
A change of variable  $ x'_7=x_7-\alpha_5 x_3 $ and relabeling we have
\begin{align*}
 &[x_{1},x_{7}]=\alpha_{1}x_8,& & [x_{2},x_{7}]= \alpha_{2}x_8, &  & [x_{3},x_{7}]=0,   \cr
 &[x_{4},x_{7}]=\alpha_{4}x_8,& &[x_{5},x_{7}]=0.
\end{align*}
Put $ x'_7=x_7-\alpha_1 x_4, $  we have
\begin{align*}
 &[x_{1},x_{7}]=0,& & [x_{2},x_{7}]= \alpha_{2}x_8, &  & [x_{3},x_{7}]=0,   \cr
 &[x_{4},x_{7}]=\alpha_{4}x_8,& &[x_{5},x_{7}]=0.
\end{align*}
Therefore 
\[L=\langle x_1, \dots, x_8 \mid  [x_1, x_2]=x_6, [x_1, x_4]=x_8, [x_3, x_5]=x_8, [x_{2},x_{7}]= \alpha_{2}x_8, [x_{4},x_{7}]=\alpha_{4}x_8 \rangle. 
\]
Consider the  following cases.\\
Case 1. Let $ \alpha_2\neq 0 $ and $ \alpha_{4}=0. $ Then by using a change of variable $ x'_7=\alpha_2 x_7 $ and relabeling, we have
\[S_1=\langle x_1, \dots, x_8 \mid  [x_1, x_2]=x_6, [x_1, x_4]=x_8, [x_3, x_5]=x_8, [x_{2},x_{7}]= x_8 \rangle.
\]
Case 2.  Let $ \alpha_2= 0 $ and $ \alpha_{4}\neq 0. $ A change of variable $ x'_7=\alpha_4 x_7 $ and relabeling we have
\begin{equation*}
  [x_1, x_2]=x_6, [x_1, x_4]=x_8, [x_3, x_5]=x_8, [x_{4},x_{7}]=x_8. 
\end{equation*}
Put $ x'_1=x_1+x_7. $ By relabeling we have
\[\langle x_1, \dots, x_8 \mid  [x_1, x_2]=x_6, [x_3, x_5]=x_8, [x_{4},x_{7}]=x_8 \rangle \cong H(1)\oplus H(2).
\]
  Case 3. Let $ \alpha_2\neq 0 $ and $ \alpha_{4}\neq 0. $ Then by using $ x'_7=\alpha_4^{-1}x_7 $ and relabeling, we have 
 \begin{equation*}
 [x_1, x_2]=x_6, [x_1, x_4]=x_8, [x_3, x_5]=x_8, [x_{2},x_{7}]= \alpha_{2}x_8, [x_{4},x_{7}]=x_8.
 \end{equation*} 
 Therefore by a change of variable $ x'_2=\alpha_2^{-1}x_2, $ $ x'_6= \alpha_2^{-1} x_6$ and relabeling we have
 \begin{equation*}
  [x_1, x_2]=x_6, [x_1, x_4]=x_8, [x_3, x_5]=x_8, [x_{2},x_{7}]=x_8, [x_{4},x_{7}]=x_8.
 \end{equation*}
Finally, appropriate   changes of variables $ x'_1=x_1+x_7, $  $ x'_2=x_2-x_4 $ and $ x'_6=x_6-x_8 $ and relabeling show that 
 \[\langle x_1, \dots, x_8 \mid  [x_1, x_2]=x_6,  [x_3, x_5]=x_8,  [x_{4},x_{7}]=x_8 \rangle \cong H(1)\oplus H(2).
\]
Therefore  $ L $ is isomorphic to $ H(1)\oplus H(2) $ or $ S_1 $ in this case, as required.
\end{proof}
\begin{thm}\label{th3.4}
Let $ L $ be an $ n $-dimensional nilpotent Lie algebra with $ s(L)=7$ and $ \dim L^{2}=2. $ Then $ L $ is isomorphic to  $ L_{5,8}\oplus A(6) $ $ L_{4,3}\oplus A(5), $ $ L_{5,5}\oplus A(4), $ $ L_{6,22}(\varepsilon) \oplus A(4) $ for all $ \varepsilon \in \mathbb{F}/(\stackrel{*}{\sim}), $  $ 27B, $  $ L_{6,10}\oplus A(2),$    $ 27A\oplus A(1), $ $ 157\oplus A(1), $   $ L_{6,10}\dotplus H(1),  $ $ H(1)\oplus H(2) $ or $ S_1. $
\end{thm}
\begin{proof}
If $ L $  is a capable Lie algebra, by using a proof  similar that be used in \cite[Theorem 2.6]{s4},  $ L $ is isomorphic to one of the Lie algebras  $ L_{5,8}\oplus A(6) $ $ L_{4,3}\oplus A(5), $ $ L_{5,5}\oplus A(4), $  $ L_{6,22}(\varepsilon) \oplus A(4) $ for all $ \varepsilon \in \mathbb{F}/(\stackrel{*}{\sim})$ or $ 27B. $ 
\\
If $ L $  is a non-capable Lie algebra, then $ n\leq 9 $ by using Lemma \ref{l1.1}. Let $ n\leq 7, $ then there is no such a Lie algebra by using  Proposition \ref{pr2.3}. Let $ n=8. $ Similar to the  proof of Theorem \ref{th2.4}, we have  $ L/Z^{\wedge}(L)\cong H(1)\oplus A(4) $ and  $ s(L)=7. $ Now, by using  Proposition \ref{pr3.2}, $ L $ is isomorphic to one of the Lie algebras $ L_{6,10}\oplus A(2), $ $ 27A\oplus A(1), $ $ 157\oplus A(1), $   $ L_{6,10}\dotplus H(1),  $ $ H(1)\oplus H(2) $ or $ S_1. $ \\
If  $ n=9, $ then $ L/Z^{\wedge}(L) $ is isomorphic to $ A(6) $ or $ H(1)\oplus A(5). $ Since by using \cite[Corollary 2.2]{N1} $L\wedge L \cong  L/Z^{\wedge}(L) \wedge  L/Z^{\wedge}(L), $ we have $ \dim L\wedge L $ is equal to $ 15 $ or $ 23, $  so $ \mathcal{M}(L) $ is isomorphic to  $ A(13) $ or $ A(21) $ by using \cite[Theorem 35 (iii)]{E}. Hence $ s(L)=16 $ or $ s(L)=8, $ respectively.   Therefore 
by our assumption there is no such a Lie algebra.  
\end{proof}
\begin{lem}\label{l3.3}
Let $ L $ be an $ n $-dimensional generalized Heisenberg of rank $ 3 $ with $ s(L)=7. $ Then $ n\leq 8. $
\end{lem}
\begin{proof}
The proof is obtained by the similar procedure as in  \cite[Lemma 2.8]{s5}
\end{proof}
\begin{prop}\label{pr3.5}
There is no  stem $ 8 $-dimensional  Lie algebra $ L $ of nilpotency class two with derived subalgebra of dimension $ 3 $ and  $ s(L)=7 $ such that $ L/I\cong L_{6,22}(\varepsilon)\oplus A(1) $ for all $ \varepsilon \in \mathbb{F}/(\stackrel{*}{\sim})$ and $ I\subseteq Z(L). $
\end{prop}
\begin{proof}
There exist ideals $ I_1/I $ and $ I_2/I $ of $ L/I $ such that  $ I_1/I\cong L_{6,22}(\varepsilon) $ for all $ \varepsilon \in \mathbb{F}/(\stackrel{*}{\sim})$ and $ I_2/I\cong A(1). $ Since $ L^{2}/I\cong (I_1^{2}+I)/I, $ we have $ L^{2}=I_1^{2}+I. $ Hence $ \dim I_1^{2}=2 $ or $ 3. $ On the other hand, $ L $ is of nilpotecy class two thus $ I_1 $ is of nilpotency class two. Let $ \dim I_1^{2}=3. $
Since $ I_1 $ is of nilpotency class two and  $ \dim I_1^{2}=3, $  $ I_1 $ is isomorphic to one of the Lie algebras $ 37A, $ $ 37B, $ $ 37C $ or $ 37D$ by looking at  Table \ref{ta4}. If $ \dim I_1^{2}=2, $ then $ I_1 $ is isomorphic to $ L_{6,22}(\varepsilon) \oplus A(1) $ for all $ \varepsilon \in \mathbb{F}/(\stackrel{*}{\sim}) $ by looking at Table \ref{ta2}. On the other hand,   $  I_2/I\cong A(1), $ we have $ I_2\cong A(2). $  We claim that and  $ [I_1, I_2]=I. $ Otherwise $ [I_1, I_2]=0 $ and so $ I_{2}\subseteq Z(L). $ Hence $ \dim Z(L)\geq 4. $ It is a contradiction.
Therefore $ L=I_{1} + I_{2} =I_{1} \rtimes A(1) $ such that $ I_1 $ is isomorphic to   one of the Lie algebras $ 37A, $ $ 37B, $ $ 37C, $  $ 37D $ or $L_{6,22}(\varepsilon) \oplus A(1) $ for all $ \varepsilon \in \mathbb{F}/(\stackrel{*}{\sim}) $ and $ [I_1, A(1)]=\langle x_{7} \rangle . $  \\
Let $ L\cong I_{1}\rtimes A(1) $ such that $ I_1\cong 37A=\langle x_1,\dots, x_6, x_8 \mid [x_1, x_2]=x_5, [x_2, x_3]=x_6, [x_2, x_4]=x_8\rangle $ and $ A(1)=\langle x_7 \rangle. $ Then the relations of $ L $ are as follow. 
\begin{align*}
&[x_1, x_2]=x_5, &&[x_2, x_3]=x_6, && [x_2, x_4]=x_8,&& [x_1, x_7]=\alpha_1 x_8,\cr
&[x_2, x_7]=\alpha_2 x_8, &&[x_3, x_7]=\alpha_3 x_8, &&[x_4, x_7]=\alpha_4 x_8.
\end{align*}
By using a change of variable $ x'_7=x_7-\alpha_2x_4 $ and relabiling, we have
\begin{align*}
&[x_1, x_2]=x_5, &&[x_2, x_3]=x_6, && [x_2, x_4]=x_8,&& [x_1, x_7]=\alpha_1 x_8,\cr
&[x_2, x_7]=0, &&[x_3, x_7]=\alpha_3 x_8, &&[x_4, x_7]=\alpha_4 x_8.
\end{align*}
 Since $ [I_1, A(1)]=\langle x_{8} \rangle, $ at least one of the $ \alpha_1, $ $ \alpha_3 $ or $ \alpha_4 $ is non-zero. Let $ \alpha_1 \neq 0. $ Then we are going to obtain  the Schur multiplier of $ L. $ By using the method of Hardy and Stitzinger in \cite{B1}. Start with
 \begin{align*}
  &[x_{1},x_{2}]=x_{5}+s_{1},& & [x_{1},x_{3}]= s_{2}, &  & [x_{1},x_{4}]=s_{3},   \cr
 &[x_{1},x_{5}]=s_{4},& &[x_{1},x_{6}]=s_{5},& &    [x_{1},x_{7}]=\alpha_1 x_8+s_{6},  \cr
  &[x_{1},x_{8}]=s_{7}, &&[x_{2},x_{3}]=x_{6}+s_{8},& &[x_{2},x_{4}]=x_8+s_{9}, \cr
 &[x_{2},x_{5}]=s_{10}, &&[x_{2},x_{6}]=s_{11},&&[x_{2}, x_{7}]=s_{12},\cr
& [x_{2},x_{8}]=s_{13}, &&[x_{3},x_{4}]=s_{14}, &&[x_{3},x_{5}]=s_{15},\cr
&[x_{3}, x_{6}]=s_{16}, && [x_{3},x_{7}]=\alpha_3 x_8+s_{17},&& [x_{3},x_{8}]=s_{18},\cr
 &[x_{4}, x_{5}]=s_{19},&& [x_{4}, x_{6}]=s_{20}, && [x_{4},x_{7}]=\alpha_4 x_8+s_{21}, \cr
  &[x_{4},x_{8}]=s_{22},&&[x_{5},x_{6}]=s_{23},&& [x_{5}, x_{7}]=s_{24},\cr
  &[x_{5},x_{8}]=s_{25}, && [x_{6},x_{7}]=s_{26}, && [x_{6},x_{8}]=s_{27},\cr
  & [x_{7},x_{8}]=s_{28},
\end{align*}
where $  s_{1},..., s_{28} $ generate $ \mathcal{M}(L) $. Put $x^{\prime}_{5}=x_{5}+s_{1},$  $x^{\prime}_{6}=x_{6}+s_{8}$ and  $x^{\prime}_{8}=x_{8}+s_{9}.$   A change of variables allows that $ s_{1}=s_{8}=s_{9}=0. $ Now, by using the Jocobi identity we have
\begin{align*}
&s_{5}=[x_{1}, x_{6}]=[x_{1},[x_{2},x_{3}]]=-([x_{3},[x_{1},x_{2}]]+[x_{2},[x_{3},x_{1}]])=-s_{15},\cr
&s_7=[x_{1}, x_{8}]=[x_{1},[x_{2},x_{4}]]=-([x_{4},[x_{1},x_{2}]]+[x_{2},[x_{4},x_{1}]])=-s_{19},\cr
&s_{18}=[x_{3}, x_{8}]=[x_{3},[x_{2},x_{4}]]=-([x_{4},[x_{3},x_{2}]]+[x_{2},[x_{4},x_{3}]]= s_{20},\cr
&s_{18}=[x_{3}, x_{8}]=[x_{3},\frac{1}{\alpha_1}[x_{1},x_{7}]]=-\frac{1}{\alpha_1}([x_{7},[x_{3},x_{1}]]+[x_{1},[x_{7},x_{3}]])=\frac{\alpha_3}{\alpha_1}s_7,\cr
&s_{22}=[x_{4}, x_{8}]=[x_{4},\frac{1}{\alpha_1}[x_{1},x_{7}]]=-\frac{1}{\alpha_1}([x_{7},[x_{4},x_{1}]]+[x_{1},[x_{7},x_{4}]])=\frac{\alpha_4}{\alpha_1}s_7,\cr
&s_{23}=[x_{5}, x_{6}]=[x_{5},[x_{2},x_{3}]]=-([x_{3},[x_{5},x_{2}]]+[x_{2},[x_{3},x_{5}]]=0,\cr
&s_{24}=[x_{5}, x_{7}]=-[x_{7},[x_{1},x_{2}]]=[x_{2},[x_{7},x_{1}]]+[x_{1},[x_{2},x_{7}]]=-\alpha _1 s_{13},\cr
&s_{25}=[x_{5}, x_{8}]=[x_{5},[x_{2},x_{4}]]=-([x_{4},[x_{5},x_{2}]]+[x_{2},[x_{4},x_{5}]])=0,\cr
&s_{26}=[x_{6}, x_{7}]=-[x_{7},[x_{2},x_{3}]]=[x_{3},[x_{7},x_{2}]]+[x_{2},[x_{3},x_{7}]]=\alpha _3 s_{13},\cr
&s_{27}=[x_{6}, x_{8}]=[x_{6},[x_{2},x_{4}]]=-([x_{4},[x_{6},x_{2}]]+[x_{2},[x_{4},x_{6}]])=0,\cr
&s_{28}=[x_{7}, x_{8}]=[x_{7},[x_{2},x_{4}]]=-([x_{4},[x_{7},x_{2}]]+[x_{2},[x_{4},x_{7}]])=\alpha _{4}s_{13}.\cr
\end{align*}
 Therefore
\[ \mathcal{M}(L)=\langle s_2, s_{3}, s_{4},  s_{5}, s_{6}, s_{7}, s_{10},  s_{11}, s_{12}, s_{13}, s_{14},  s_{16}, s_{17}, s_{21} \rangle
\]
and so $ \dim \mathcal{M}(L)=14. $ Hence $ s(L)\neq 7. $ Also, if  $ \alpha_{3} $ or $ \alpha_{4} $ are non-zero, then $ \dim \mathcal{M}(L)=14$ and so $ s(L)\neq 7. $\\
By using a similar way if   $ L=I_{1} \rtimes A(1) $ such that $ I_1 $ is isomorphic to   one of the Lie algebras  $ 37B, $ $ 37C, $  $ 37D $ or $L_{6,22}(\varepsilon) \oplus A(1) $ and $ [I_1, A(1)]=\langle x_{8} \rangle,  $ then $ s(L)\neq 7. $ Therefore by our assumption there is no  such a Lie algebra. 
\end{proof}
Let $ cl(L) $ be used to show the nilpotency class of $ L. $ Then 
\begin{prop}\label{pr3.6}
There is no stem $ 8 $-dimensional  Lie algebra $ L $  with derived subalgebra of dimension $ 3 $ and  $ s(L)=7 $ such that $ L/I\cong L_{5,8}\oplus A(2) $ and $ I\subseteq Z(L). $
\end{prop}
\begin{proof}
There are ideals $ I_{1}/I $ and $ I_2/I $ such that $ I_{1}/I\cong  L_{5,8}\oplus A(1) $ and $ I_2/I \cong A(1). $ Consider the following cases.\\ 
Case 1. Let $ \dim Z(L)=1. $ Then $ Z(L)=I. $ Since  $ L^{2}/Z(L)=I^{2}_{1}+Z(L)/Z(L), $ we have $ L^{2}= I^{2}_{1}+Z(L)$ and
 by using Lemma \ref{li}, $ L^{2}=I^{2}_1, $ $ Z(L)=I^{3}_1 $ and $ cl(L)=cl(I_1)=3. $ Also, $ I_2=A(1)\oplus Z(L). $
It is clear that  $ [I_1, I_2]\subseteq I_1\cap I_2 $ and  $ I_1\cap I_2 =Z(L), $ hence $  [I_1, I_2]\subseteq Z(L). $ On the other hand, $ L=I_1+I_2=I_1+A(1) $ and so $ [I_1, I_2]=[I_1, A(1)]. $ If  $  [I_1, A(1)]=0, $ then $ L=I_1\oplus A(1) $ and since $ L $ is stem, we have a contradiction. Therefore $  [I_1, A(1)]=Z(L). $ We claim that $ Z(I_{1})\cap I_{1}^{2}=Z(L), $ otherwise  there is $a\in  Z(I_{1})\cap I_{1}^{2} $ such that $ a\notin Z(L). $ Since $ a\in I_{1}^{2}, $ there are $ b,c\in I_1 $ such that $ [b,c]=a. $ Now, let $ A(1)=\langle x_8 \rangle. $ Then $ [x_8, [b,c]]=-([c, [x_8,b]]+[b, [c,x_8]]) $  by using Jacobi identity. Since $ [x_8,b] $ and $ [c,x_8] $ are central, we have $  [x_8, [b,c]]=0$ and so $ a\in Z(L). $ It is a contradiction. If $ I_1 $ is stem, then $ Z(I_1)=I_1^{3} $ and $ \dim Z(I_1)=1. $ By using the  classification of all $ 7 $-dimensional Lie algebra in \cite{Mi} $ I_1 $ is isomorphic to one of the nilpotent Lie algebras $ 147A $ or $ 147B. $
If $ I_1 $ is not  stem, then $ I_1 $  is isomorphic to one of the nilpotent Lie algebras $ L_{6,19}(\varepsilon)\oplus A(1) $ for all $ \varepsilon \in \mathbb{F}^{*}/(\stackrel{*}{\sim})$ or $ L_{6,20}\oplus A(1) $ by using Table \ref{ta5}. Now, we are going to determine $ L. $ It is sufficient to obtain  $ [I_1, A(1)]. $ Assume that $ A(1)=\langle x_8 \rangle  $ and 
\[
 I_1\cong 147A=\langle x_1, \dots, x_7\mid [x_1, x_2]=x_4, [x_1, x_3]=x_5, [x_1, x_6]=x_7,  [x_2, x_5]=x_7,
 \]
 \[
  [x_3, x_4]=x_7\rangle.
\]
 We have
\begin{align*}
&[x_1, x_8]=\alpha_1 x_7, && [x_2, x_8]=\alpha_2 x_7,&& [x_3, x_8]=\alpha_3 x_7, \cr
 & [x_4, x_8]=\alpha_4 x_7, && [x_5, x_8]=\alpha_5 x_7, &&[x_6, x_8]=\alpha_6 x_7.
\end{align*}  
Put $ x'_{8}=x_8-\alpha_1 x_6. $ Then by relabeling
\begin{align*}
&[x_1, x_8]=0, && [x_2, x_8]=\alpha_2 x_7, &&[x_3, x_8]=\alpha_3 x_7, \cr
 & [x_4, x_8]=\alpha_4 x_7, && [x_5, x_8]=\alpha_5 x_7, && [x_6, x_8]=\alpha_6 x_7.
\end{align*}  
A change of variable $ x'_{8}=x_8-\alpha_1 x_5 $ and relabeling  allows that 
\begin{align*}
&[x_1, x_8]=0, && [x_2, x_8]=0, &&[x_3, x_8]=\alpha_3 x_7, \cr
 & [x_4, x_8]=\alpha_4 x_7, && [x_5, x_8]=\alpha_5 x_7, && [x_6, x_8]=\alpha_6 x_7.
\end{align*}  
Finally, by putting $ x'_{8}=x_8-\alpha_3 x_4 $ and relabeling we have
\begin{align*}
&[x_1, x_8]=0, && [x_2, x_8]=0, && [x_3, x_8]=0, \cr
 & [x_4, x_8]=\alpha_4 x_7, &&[x_5, x_8]=\alpha_5 x_7, && [x_6, x_8]=\alpha_6 x_7.
\end{align*}
By using Jacobi identity $  [x_4, x_8]=[[x_1, x_2], x_8]=[x_2,[x_8,x_1]]+[x_1,[x_2,x_8]] $ and since $ [x_8,x_1] $ and $ [x_2,x_8] $ are central, we have  $ [x_4, x_8]=0. $ By using a similar way $  [x_5, x_8]=0. $ 
Therefore 
\begin{align*}
&[x_1, x_8]=0, && [x_2, x_8]=0, && [x_3, x_8]=0, \cr
 & [x_4, x_8]=0, && [x_5, x_8]=0, && [x_6, x_8]=\alpha_6 x_7.
\end{align*}
Since $ [I_1, A(1)]\neq 0, $ we have  $ \alpha_6\neq 0 $  and so  by using $ x'_8=\alpha_6^{-1}x_8 $ and relabeling $ [I_1, A(1)]=\langle [x_6, x_8]= x_7\rangle. $ Therefore
\[
 L\cong \langle x_1, \dots, x_8 \mid [x_1, x_2]=x_4, [x_1, x_3]=x_5, [x_1, x_6]=x_7,  [x_2, x_5]=x_7,
 \]
 \[
  [x_3, x_4]=x_7, [x_6, x_8]= x_7\rangle.
\]
By using the method of Hardy and Stitzinger in \cite{B1},  we can see $ \dim \mathcal{M}(L)=12 $ and so $ s(L)\neq 7. $ Also, if $ I_1 $ is isomorphic to one of the nilpotent Lie algebras $ 147B, $ $ L_{6,19}(\varepsilon)\oplus A(1) $ or $ L_{6,20}, $ then by using a similar way one can see  that $ s(L)\neq 7. $\\
Case 2. Let $ \dim Z(L)=2 $ and $ Z(L)\subset L^{2}. $ Then  $ L/Z(L)\cong H(1)\oplus A(3) $ and so $ \dim Z_2(L)=6. $ Also, by our assumption, we know that $ \dim L=8 $ and $ s(L)=7, $ hence $ \dim \mathcal{M}(L)=15. $ On the other hand, $ L^{3}\subseteq Z(L). $ If $ L^{3}\neq Z(L), $ then $ \dim L^{3}=1 $ and $ L/I^{3}\cong L_{5,8}\oplus A(2). $ When  $ L^{3}=Z(L), $ then $ \dim L^{3}=2 $ and $ L/I^{3}\cong H(1)\oplus A(3). $ By using \cite[Lemmas 2.3 and  2.10]{P}, \cite[Lemma 23]{M} and \cite[Theorem 1]{B1} one can see $ \dim \mathcal{M}(L_{5,8}\oplus A(2))=10 $ and $ \dim \mathcal{M}(H(1)\oplus A(3))=13. $ Now, from \cite[Theorem 2.6]{s3}, we have
\begin{equation}\label{3}
\dim L^{3}+\dim \mathcal{M}(L)\leq \dim \mathcal{M}(L/L^{3})+\dim (L/Z_{2}(L)\otimes L^{3}),
\end{equation}
which is a contradiction.\\
Case 3. Let $ \dim Z(L)=3 $ and $ Z(L)=L^{2}. $ Then $ L  $ and $  I_1 $ are of nilpotency class two.  Since $ L^{2}=I^{2}_1+I, $ we have $ \dim I^{2}_1=2 $ or $ 3. $ If $ \dim I^{2}_1=2, $ then $ \dim I^{2}_1\cap I=0 $ and so $ I\not\subset I^{2}_1. $ On the other hand, $ I\subseteq I_1 $ and $ I\subseteq Z(L), $ thus $ I \subseteq Z(I_1) $ and so $ Z(I_1)\not\subset I^{2}_1. $ Now, by looking at Table \ref{ta2}, $ I_1 $ is isomorphic to $ L_{5,8}\oplus A(2). $ Similar to Case 1, we have 
\[
L\cong \langle x_1, \dots, x_8 \mid [x_1, x_2]=x_4, [x_1, x_3]=x_5, [x_6, x_8]=x_7 \rangle. 
\] 
By using the method of Hardy and Stitzinger in \cite{B1}, one can see that $ \dim \mathcal{M}(L)=14 $ and so $ s(L)\neq 7. $\\
If $ \dim I^{2}_{1}=3,$ then  $ I_1 $ is isomorphic to one of the nilpotent Lie algebras $ 37A, $ $ 37B $ or $ 37C $ by using Table \ref{ta7}. By using a similar way  we may obtain that  $ \mathcal{M}(L)=14. $ Hence  $ s(L)\neq 7. $ Therefore there is no such a Lie algebra.
\end{proof}
\begin{thm}\label{th3.6}
Let $ L $ be an $ n $-dimensional nilpotent Lie algebra with $ s(L)=7 $ and $ \dim L^{2}=3. $ Then $ L $ is isomorphic to one of the nilpotent Lie algebras  $ L_{6,23}\oplus A(1), $ $ L_{6,25}\oplus A(1), $ $ 37B\oplus A(1), $ $ 37C\oplus A(1), $ $ 37D\oplus A(1), $  $ L_{6,26}\oplus A(2), $  $ L_{6,13}, $ $ 257A, $ $ 257C $ or $ 257F. $
\end{thm}
\begin{proof}
If $ Z(L)\not\subseteq L^{2}, $ then there is a central one dimensional ideal $ I $ of $ L $ such that $ I\cap L^{2}=0. $ By using a similar method in the proof of Theorem \ref{th2.5}, we can see  $ L $ is isomorphic to $ L_{6,23}\oplus A(1), $ $ L_{6,25}\oplus A(1), $ $ 37B\oplus A(1), $ $ 37C\oplus A(1), $ $ 37D\oplus A(1) $ or $ L_{6,26}\oplus A(2). $ Also, by using Tables \ref{ta6}, \ref{ta7}, \cite[Lemma 23]{M} and \cite[Theorem 1]{B1} we have $ s(L_{6,23}\oplus A(1))=s( L_{6,25}\oplus A(1))= s( 37B\oplus A(1)) =s(37C\oplus A(1)) =s( 37D\oplus A(1)) = s(L_{6,26}\oplus A(2))=7. $\\
If $ Z(L)\subseteq L^{2}, $ then there is a one dimensional ideal $ I $ of $ L $ such that $ I\subseteq Z(L)\cap L^{2}. $ Hence by using we have
\begin{align*}
\frac{1}{2}(n-1)(n-2)-5\leq \frac{1}{2}(n-2)(n-3)+1-s(L/I)+n-3
\end{align*}
and so $ s(L/I)\leq 5. $ Since $ \dim (L/I)^{2}=2, $ by using  Theorem \ref{s},  we have one of the following cases.\\
Case 1. If $ L/I $ is isomorphic to one of the nilpotent Lie algebras $ L_{5,8}, $ $ L_{4,3}, $ $ L_{5,8}\oplus A(1), $ $ L_{4,3}\oplus A(1), $ $ L_{5,5}, $ $ L_{6,22}(\varepsilon)$ for all $ \varepsilon \in \mathbb{F}/(\stackrel{*}{\sim}) $ or $ L_{4,3}\oplus A(2), $ 
then $ n\leq 7. $ Hence $ L$ is isomorphic to one of the nilpotent Lie algebras $L_{6,13}, $ $ 257A, $ $ 257C $ or $ 257F $  by looking at Tables \ref{ta6} and \ref{ta7}. \\
Case 2. If $ L/I\cong L_{5,8}\oplus A(2), $ then  there is no Lie algebra $ L $ with $ s(L)=7 $ and $ \dim L^{2}=3 $ by using  Proposition \ref{pr3.6}.\\
Case 3. Let $L/I\cong L_{6,22}(\varepsilon)\oplus A(1)$ for all $ \varepsilon \in \mathbb{F}/(\stackrel{*}{\sim}). $ If $ L $ is of nilpotency class two, then  there is no Lie algebra with $ s(L)=7 $ and $ \dim L^{2}=3 $ by using Theorem \ref{pr3.5}. If $ L $ is of nilpotency class $ 3, $ then 
\begin{align*}
15=\dim \mathcal{M}(L) & \leq \dim \mathcal{M}(L/\gamma_{3}(L))-\dim \gamma_{3}(L)+\dim L/\gamma_{2}(L)\otimes \gamma_{3}(L)-2\cr
& \leq 14.
\end{align*}
by using \eqref{1} and \eqref{2}. It is a contradiction.
\\ 
Case 4. Let $ L/I $  is isomorphic to one of the nilpotent Lie algebras $ L_{5,8}\oplus A(3), $  $ L_{5,8}\oplus A(4)  $  or $ L_{6,22}(\varepsilon)\oplus A(2)$ for all $ \varepsilon \in \mathbb{F}/(\stackrel{*}{\sim}). $ If $ L $ is of nilpotency class two,    then there is no Lie algebra with $ s(L)=7 $ and $ \dim L^{2}=3 $ by using Lemma \ref{l3.3}. If $ L $ is of nilpotency class $ 3, $  then  we have a contradiction  similar to  the Case $ 3. $\\
Case 5. Let $ L/I $  is isomorphic to $ L_{4,3}\oplus A(3), $ $  L_{5,5}\oplus A(2). $ If $ L/I\cong L_{4,3}\oplus A(3), $  then $ \dim \mathcal{M}(L/I)=11 $ by using \cite[Lemma 2.9]{P}, \cite[Lemma 23]{M} and \cite[Theorem 1]{B1}. When $ cl(L)=3 $ and $ \dim \gamma_3(L)=1, $ then by using the inequalities  \eqref{1} and \eqref{2} we have a contradiction. In the case $ cl(L)=3 $ and $ \dim \gamma_3(L)=2, $ by using  the inequality \eqref{3},  we get a contradiction. If $ cl(L)=4, $ then  $ I= \gamma_4(L)$  and  $ \dim \gamma_4(L)=1. $ Since  $ \dim \mathcal{M}(L)=15 $ and  $ \dim \mathcal{M} (L/ \gamma_3(L))=11, $ by using \cite[proof of Theorem 1.1]{R},  we have 
 \begin{equation*}
1=n-m-c \leq\dim \ker \lambda_{4}, ~~ \text{and} 
\end{equation*}
\begin{equation*}
\dim \mathcal{M}(L)=\dim \mathcal{M}(L/\gamma_{4}(L))-\dim \gamma_{4}(L)+\dim L/\gamma_{2}(L)\otimes \gamma_{4}(L)-\dim \ker \lambda_{4}
\end{equation*}
which a  contradiction. If $ L/I\cong L_{5,5}\oplus A(2), $ by using a similar way we get a contradiction. 
\end{proof}
\begin{prop}\label{pr3.7}
There is no  stem $ 8 $-dimensional nilpotent Lie algebra $ L $ of nilpotency class $ 2 $ with derived subalgebra of dimension $ 4 $ and   $ s(L)=7 $ such that $ L/I\cong 37A $ and $ I\subseteq Z(L). $
\end{prop}
\begin{proof}
Let $ I=\langle x_8 \rangle $ and $ 37A=\langle x_1, \dots, x_6, x_7 \mid [x_1, x_2]=x_5,  [x_2, x_3]=x_6, [x_2, x_4]=x_7 \rangle.$ Then
\begin{align*}
&[x_{1},x_{2}]=x_{5}+\alpha _{1}x_{8},&&  [x_{1},x_{3}]=\alpha _{2}x_{8}, &   & [x_{1},x_{4}]=\alpha _{3}x_{8},   \cr
&[x_{1},x_{5}]=\alpha _{4}x_{8},&& [x_{1},x_{6}]=\alpha _{5}x_{8},&& [x_{1},x_{7}]=\alpha _{6}x_{8},\cr
& [x_{2},x_{3}]=x_6+\alpha _{7}x_{8},    &  & [x_{2},x_{4}]=x_7+\alpha _{8}x_{8},&& [x_{2},x_{5}]=\alpha _{9}x_{8},\cr
  &[x_{2},x_{6}]=\alpha _{10}x_{8}, && [x_{2},x_{7}]=\alpha _{11}x_{8}, &&[x_{3},x_{4}]=\alpha _{12}x_{8},\cr
 &[x_{3},x_{5}]=\alpha _{13}x_{8}, & &  [x_{3},x_{6}]=\alpha _{14}x_{8}, && [x_{3},x_{7}]=\alpha _{15}x_{8}, \cr
&[x_{4},x_{5}]=\alpha _{16}x_{8}, & &  [x_{4},x_{6}]=\alpha _{17}x_{8}, && [x_{4},x_{7}]=\alpha _{18}x_{8},\cr
&[x_{5},x_{6}]=\alpha _{19}x_{8}, & &  [x_{5},x_{7}]=\alpha _{20}x_{8}, && [x_{6},x_{7}]=\alpha _{21}x_{8}.
\end{align*}
Since $ L $ is of nilpotency class two, we have
 $ \alpha_4=\alpha_5=\alpha_6=\alpha_9=\alpha_{10}=\alpha_{11}=\alpha_{13}=\alpha_{14}=\alpha_{15}=\alpha_{16}=
\alpha_{17}=\alpha_{18}=\alpha_{19}=\alpha_{20}=\alpha_{21}=0. $ By using a change of variables  $ x'_5=x_{5}+\alpha _{1}x_{8}, $ $ x'_6=x_6+\alpha _{7}x_{8} $ and $ x'_7= x_7+\alpha _{8}x_{8},$ we have $ \alpha_1=\alpha_7=\alpha_8=0. $
Since  $ \dim L^{2}=4, $  at least one of the coefficients $ \alpha _{2}, $ $ \alpha _{3}$ or $ \alpha _{12} $ is non-zero. \\
Let $ \alpha_2\neq 0.$ Then  by using $ x'_8=\alpha _{2}x_{8} $ and relabeling we have
\[
L=\langle x_1, \dots, x_7, x_8 \mid   [x_{1},x_{2}]=x_{5}, [x_{2},x_{3}]=x_6,  [x_{2},x_{4}]=x_7, [x_{1},x_{3}]=x_{8}, [x_{1},x_{4}]=\alpha _{3}x_{8}, \]
 \[[x_{3},x_{4}]=\alpha _{12}x_{8} \rangle.
\]
In the following, we are going to obtain  the Schur multiplier of $ L, $ by using the method of Hardy and Stitzinger in \cite{B1}. Start with
\begin{align*}
  &[x_{1},x_{2}]=x_{5}+s_{1},& & [x_{1},x_{3}]= x_{8}+s_{2}, &  & [x_{1},x_{4}]=\alpha _3 x_{8}+s_{3},   \cr
 &[x_{1},x_{5}]=s_{4},& &[x_{1},x_{6}]=s_{5},& &    [x_{1},x_{7}]=s_{6},  \cr
  &[x_{1},x_{8}]=s_{7}, &&[x_{2},x_{3}]=x_{6}+s_{8},& &[x_{2},x_{4}]=x_7+s_{9}, \cr
 &[x_{2},x_{5}]=s_{10}, &&[x_{2},x_{6}]=s_{11},&&[x_{2}, x_{7}]=s_{12},\cr
& [x_{2},x_{8}]=s_{13}, &&[x_{3},x_{4}]=\alpha_{6}x_8+s_{14}, &&[x_{3},x_{5}]=s_{15},\cr
&[x_{3}, x_{6}]=s_{16}, && [x_{3},x_{7}]=s_{17},&& [x_{3},x_{8}]=s_{18},\cr
 &[x_{4}, x_{5}]=s_{19},&& [x_{4}, x_{6}]=s_{20}, && [x_{4},x_{7}]=s_{21}, \cr
  &[x_{4},x_{8}]=s_{22},&&[x_{5},x_{6}]=s_{23},&& [x_{5}, x_{7}]=s_{24},\cr
  &[x_{5},x_{8}]=s_{25}, && [x_{6},x_{7}]=s_{26}, && [x_{6},x_{8}]=s_{27},\cr
  & [x_{7},x_{8}]=s_{28},
\end{align*}
where $  s_{1},..., s_{28} $ generate $ \mathcal{M}(L) $.
By using the Jocobi identity, we have
\begin{align*}
&s_{5}=[x_{1}, x_{6}]=[x_{1},[x_{2},x_{3}]]=-([x_{3},[x_{1},x_{2}]]+[x_{2},[x_{3},x_{1}]])=-s_{15}+s_{13},\cr
&s_6=[x_{1}, x_{7}]=[x_{1},[x_{2},x_{4}]]=-([x_{4},[x_{1},x_{2}]]+[x_{2},[x_{4},x_{1}]])=-s_{19}+\alpha_{3}s_{13},\cr
&s_{15}=[x_{3}, x_{5}]=[x_{3},[x_{1},x_{2}]]=-([x_{2},[x_{3},x_{1}]]+[x_{1},[x_{2},x_{3}]])= s_{13}-s_5,\cr
&s_{17}=[x_{3}, x_{7}]=[x_{3},[x_{2},x_{4}]]=-([x_{4},[x_{3},x_{2}]]+[x_{2},[x_{4},x_{3}]])=s_{20}+\alpha_6 s_{13},\cr
&s_{19}=[x_{4}, x_{5}]=[x_{4},[x_{1},x_{2}]]=-([x_{2},[x_{4},x_{1}]]+[x_{1},[x_{2},x_{4}]])=\alpha_3 s_{13}-s_{6},\cr
&s_{22}=[x_{4}, x_{8}]=[x_{4},[x_{1},x_{3}]]=-([x_{3},[x_{4},x_{1}]]+[x_{1},[x_{3},x_{4}]])=\alpha_3 s_{18}-\alpha_6 s_7,\cr
&s_{23}=[x_{5}, x_{6}]=[x_{5},[x_{2},x_{3}]]=-([x_{3},[x_{5},x_{2}]]+[x_{2},[x_{3},x_{5}]])=0,\cr
&s_{24}=[x_{5}, x_{7}]=[x_{5},[x_{2},x_{4}]]=-([x_{4},[x_{5},x_{2}]]+[x_{2},[x_{4},x_{5}]])=0,\cr
&s_{25}=[x_{5}, x_{8}]=-([x_{8},[x_{1},x_{2}]])=[x_{2},[x_{8},x_{1}]]+[x_{1},[x_{2},x_{8}]]=0,\cr
&s_{26}=[x_{6}, x_{7}]=[x_{6},[x_{2},x_{4}]]=-([x_{4},[x_{6},x_{2}]]+[x_{2},[x_{4},x_{6}]])=0,\cr
&s_{27}=[x_{6}, x_{8}]=-([x_{8},[x_{2},x_{3}]])=[x_{3},[x_{8},x_{2}]]+[x_{2},[x_{3},x_{8}]]=0,\cr
&s_{28}=[x_{7}, x_{8}]=-([x_{8},[x_{2},x_{4}]])=[x_{4},[x_{8},x_{2}]]+[x_{2},[x_{4},x_{8}]]=0,\cr
\end{align*}
Put $x^{\prime}_{5}=x_{5}+s_{1},$  $x^{\prime}_{6}=x_{6}+s_{8},$   $x^{\prime}_{7}=x_{7}+s_{9}$  and $ x'_8=x_{8}+s_{2}. $ A change of variables allows that $ s_{1}=s_2=s_{8}=s_{9}=0. $ Thus 
\[ \mathcal{M}(L)=\langle s_{3}, s_{4},  s_{5}, s_{6},s_{7}, s_{10},  s_{11}, s_{12},s_{13}, s_{14},  s_{16}, s_{17},s_{18},s_{21} \rangle
\]
 and  $\dim \mathcal{M}(L )=14$ and so $ s(L)=8. $\\
In the case $ \alpha_3 $ or $ \alpha_{12} $ is non-zero, we have $ s(L)=8 $ by using a similar way. Therefore there is no  such a nilpotent Lie algebra, and the proof is completed.
\end{proof}
\begin{thm}\label{th3.8}
Let $ L $ an  $n$-dimensional nilpotent Lie algebra with $ s(L)=7 $ and $ \dim L^{2}=4. $ Then $ L $ is isomorphic to $ L_{6,21}(\varepsilon) $ for all $ \varepsilon \in \mathbb{F}^{*}/(\stackrel{*}{\sim}). $
\end{thm}
\begin{proof}
If $ Z(L)\not\subseteq L^{2}, $ then there is a central one dimensional ideal $ I $ of $ L $ such that $ I\cap L^{2}=0. $ Similar to the proof of Theorem \ref{th2.6}, we have $ s(L/I)\leq 4. $ Since   $ \dim (L/I)^{2}=4, $ we have  there is no such a nilpotent Lie algebra $ L/I $ with $ s(L/I)\leq 4$ by using Theorem \ref{s}.\\
Let $ Z(L)\subseteq L^{2}. $ Then   there exist a  one dimension ideal $ I $ of $ L $ such that $ I\subseteq Z(L)\cap L^{2}.$ By using Theorem \ref{th1.2}, we have
\begin{align*}
\frac{1}{2}(n-1)(n-2)-5\leq \frac{1}{2}(n-2)(n-3)+1-s(L/I)+n-4
\end{align*}
and so $ s(L/I)\leq 4. $ Since $ s(L/I)\leq 4 $ and $ \dim(L/I)^{2}=3, $ we have $ L/I$ is isomorphic to $ L_{6,26,} $ $ L_{5,6}, $ $ L_{5,7}, $  $ L_{5,9} $ or $ 37A. $ If $ L/I $ is isomorphic to $ L_{6,26,} $ $ L_{5,6}, $ $ L_{5,7} $ or  $ L_{5,9}, $ then $ n=6 $ or $ 7. $
If $ n=6, $ then $ L\cong L_{6,21}(\varepsilon) $ for all $ \varepsilon \in \mathbb{F}^{*}/(\stackrel{*}{\sim}) $ by using Table \ref{ta81}.
Let $ L $ a $ 7 $-dimensional Lie algebra. Then in this case  there is no  nilpotent Lie algebra with $ s(L)=7 $ and $ \dim L^{2}=4 $ by using  Lemma \ref{l2.8}.\\
Let $ L/I\cong 37A. $ If $ L $ is of nilpotency class two, then  by using Proposition \ref{pr3.7} there is no such a Lie algebra with $ \dim L^{2}=4 $ and $ s(L)=7. $\\
If $ L $ is of nilpotency class $ 3 $, then by using \eqref{1} and \eqref{2}, we have
\begin{align*}
15=\dim \mathcal{M}(L) & \leq \dim \mathcal{M}(L/\gamma_{3}(L))-\dim \gamma_{3}(L)+\dim (L/\gamma_{2}(L)) \otimes \gamma_{3}(L)-2\cr
& \leq 14,
\end{align*}
 which is contradiction. The result follows. 
\end{proof}
We are ready to have the Main Theorem of this section.
\begin{thm}
Let $ L $ be an $ n $-dimensional nilpotent Lie algebra with $ s(L)=7. $ Then $ L $ is isomorphic to one of the nilpotent Lie algebras 
$ L_{5,8}\oplus A(6) $ $ L_{4,3}\oplus A(5), $ $ L_{5,5}\oplus A(4), $ $ L_{6,22}(\varepsilon) \oplus A(4) $ for all $ \varepsilon \in \mathbb{F}/(\stackrel{*}{\sim}), $  $ 27B, $  $ L_{6,10}\oplus A(2),$    $ 27A\oplus A(1), $ $ 157\oplus A(1), $   $ L_{6,10}\dotplus H(1),  $ $ H(1)\oplus H(2), $  $ S_1, $ $ L_{6,23}\oplus A(1), $ $ L_{6,25}\oplus A(1), $ $ 37B\oplus A(1), $ $ 37C\oplus A(1), $ $ 37D\oplus A(1), $  $ L_{6,26}\oplus A(2), $  $ L_{6,13}, $ $ 257A, $ $ 257C, $ $ 257F $ or $ L_{6,21}(\varepsilon) $ for all $ \varepsilon \in \mathbb{F}^{*}/(\stackrel{*}{\sim}). $
\end{thm}
\begin{proof}
By using Proposition \ref{pr3.1}, Theorems \ref{th3.4}, \ref{th3.6} and \ref{th3.8}, we can obtain the result. 
\end{proof}

\end{document}